\documentclass{aims}
\usepackage{amsmath}
  \usepackage{paralist}
  \usepackage{graphics} 
  \usepackage{epsfig} 
 \usepackage[colorlinks=true]{hyperref}

\allowdisplaybreaks[4]

\usepackage{graphicx,amssymb,mathrsfs,latexsym,amsthm,dsfont,bbm}
\usepackage{verbatim}

\hypersetup{urlcolor=blue, citecolor=red}

\def\currentvolume{X}
 \def\currentissue{X}
  \def\currentyear{201X}
   \def\currentmonth{XX}
    \def\ppages{X--XX}

\newtheorem{theorem}{Theorem}[section]
\newtheorem{corollary}{Corollary}

\newtheorem{lemma}[theorem]{Lemma}

\theoremstyle{definition}
\newtheorem{definition}[theorem]{Definition}
\newtheorem{remark}{Remark}

\newtheorem*{claim}{Claim}

\makeatletter

\newcommand{\Rmnum}[1]{\expandafter\@slowromancap\romannumeral #1@}
\makeatother

\title[Random bundle transformations]
      {A local variational principle for random bundle transformations}

\author[Xianfeng Ma and ercai Chen]{}

  \email{xianfengma@gmail.com}
 \email{ecchen@njnu.edu.cn}

\thanks{
The authors are  supported by the National Natural
Science Foundation of China (Grant No. 10971100).
The first author is supported by the Fundamental Research Funds for
the Central Universities.
The second author is partially supported by National Basic Research Program of China (973 Program) (Grant No. 2007CB814800)
}

\begin{document}
\maketitle

\centerline{\scshape Xianfeng Ma }
\medskip
{\footnotesize
   \centerline{Department of Mathematics}
   \centerline{ East China University of Science and Technology, Shanghai 200237, China}
} 

\medskip

\centerline{\scshape Ercai Chen}
\medskip
{\footnotesize
 \centerline{School of Mathematical Science}
   \centerline{Nanjing Normal University, Nanjing 210097, China}
   \centerline{and}
   \centerline{Center of Nonlinear Science}
   \centerline{Nanjing University, Nanjing 210093, China}
} %

\bigskip


\begin{abstract}
We introduce local topological entropy $h_{\text{top}}(T, \mathcal{U})$ and two kinds of local measure-theoretic entropy $h_{\mu}^{(r)-}(T,\mathcal{U})$ and $h_{\mu}^{(r)+}(T,\mathcal{U})$ for random bundle transformations. We derive a variational inequality of random local entropy for $h_{\mu}^{(r)+}(T,\mathcal{U})$.
As an application of such relation we prove a local variational principle in random dynamical system.

\end{abstract}


\section{Introduction}

Since the introduction of measure-theoretical entropy for an invariant measure  \cite{Kol} and topological entropy  \cite{AKM}, the relationship between these two kinds of entropy has gained a lot of attention. By the work of Goodwyn \cite{Gwyn} and Goodman \cite{Gman}, the classical variational principle was completed, namely,
\begin{equation*}
\sup_{\mu}h_{\mu}(\varphi)=h_{\text{top}}(\varphi),
\end{equation*}
where $\varphi$ is a homeomorphism from a compact metric space $X$ to itself, and the supremum is taken over all invariant measures.
A short proof for it was given by Misiurewicz \cite{Misiurewicz}.
For a factor map between two dynamical systems $(X,\varphi)$ and $(Y,\phi)$ the notions of relative topological entropy $h_{\text{top}}(\varphi,X\mid Y)$ and relative measure-theoretical entropy $h_{\mu}(\varphi,X\mid Y)$ for an invariant measure were also introduced \cite{LedWal}. Ledrappier {\it et al.} \cite{LedWal} and Downarowicz {\it et al.} \cite{DS} obtained the relative variational principle
\begin{equation*}
\sup_{\mu}h_{\mu}(\varphi,X\mid Y)=h_{\text{top}}(\varphi,X\mid Y).
\end{equation*}
The random topological entropy has been studied by Kifer \cite{Kifer} for the independent identically distributed random transformations. For  the random transformations $T$  and the corresponding skew product transformation $\Theta:\Omega\times X\rightarrow \Omega\times X$, he suggested the following relation between random measure-theoretical entropy and random topological entropy:
\begin{equation*}
\sup\{  h_{\mu}^{(r)}(T): \mu \,\,\text{is}\,\, \Theta \text{-invariant} \}=h_{\text{top}}(T),
\end{equation*}
where  $h_{\mu}^{(r)}(T)$ and $h_{\text{top}}(T)$ are the random measure-theoretical entropy and random topological entropy, respectively.
This result was extended by Bogensch{\"u}tz \cite{Bogen} to random transformations acting on one space.
Kifer \cite{Kifer2001} formulated this variational principle in full generality for random bundle transformations.

The entropy concept can be localized by defining entropy pairs or tuples both in measure-theoretical and topological situations \cite{GY}. To study the relationship between the two kinds of entropy pairs or tuples, one needs a local version of the variational principle. Blanchard {\it et al.} \cite{Blanchard1997} showed that for a given topological dynamical system $(X,\varphi)$ and an open cover $\mathscr{U}$ of $X$ there exists an invariant measure $\mu$ with $\inf_{\alpha}h_{\mu}(\varphi,\alpha)\geq h_{\text{top}}(\varphi,\mathscr{U})$, where the infimum is taken over all partitions of $X$ which are finer than $\mathscr{U}$. Huang {\it et al.} \cite{huangye2006} provided some kind of converse statement of this result. To study the question of whether the two quantities is equal or not, Romagnoli \cite{Rom2003} introduced two kinds of measure-theoretical entropy $h_{\mu}^{-}$ and $h_{\mu}^{+}$ for covers. He showed that for a topological dynamical system $(X,\varphi)$ there is an invariant measure $\mu$ with $h_{\mu}^{-}(\varphi, \mathscr{U})=h_{\text{top}}(\varphi,\mathscr{U})$, and consequently obtained the local variational principle for a given open cover, i.e.
\begin{equation*}
\max_{\mu}h_{\mu}^{-}(\varphi, \mathscr{U})=h_{\text{top}}(\varphi,\mathscr{U}).
\end{equation*}
For a factor map between two topological dynamical systems  $(X,\varphi)$ and $(Y,\phi)$, Huang {\it et al.} \cite{Huang2006} introduced two notions of measure-theoretical conditional entropy for covers, namely $h_{\mu}^-(\varphi,\mathscr{U}\mid Y)$ and $h_{\mu}^+(\varphi,\mathscr{U}\mid Y)$. They showed that for the factor map and a given open cover $\mathscr{U}$ of $X$, the local relative variational principle holds, i.e.
\begin{equation*}
\max_{\mu}h_{\mu}^{-}(\varphi, \mathscr{U}\mid Y)=h_{\text{top}}(\varphi,\mathscr{U}\mid Y).
\end{equation*}
We remark that the classical variational principle could follow from the local ones or the relative ones by some simple arguments.

Now it is a natural question if there exists a local variational principle for random bundle transformations. We will address this question in the current paper.

To study the question we introduced two notions of random measure-theoretical entropy for covers in random dynamical system, namely $h_{\mu}^{(r)-}(T,\mathcal{U})$ and $h_{\mu}^{(r)+}(T,\mathcal{U})$.
We derive  a variational inequality of random entropy for $h_{\mu}^{(r)+}(T,\mathcal{U})$, i.e.
for a given open cover $\mathcal{U}$ of the measurable subset $\mathcal{E}\subset \Omega \times X$, there always exists a  $\Theta$-invariant measure $\mu $ such that $h_{\mu}^{(r)+}(T,\mathcal{U})\geq h_{\text{top}}(T,\mathcal{U})$. Moreover, Using this variational inequality, we could show that for a given open cover $\mathcal{U}$,
\begin{equation*}
\max\{  h_{\mu}^{(r)-}(T,\mathcal{U}): \mu \,\,\text{is}\,\, \Theta \text{-invariant} \}=h_{\text{top}}(T,\mathcal{U}),
\end{equation*}
The classical variational principle for random bundle transformations follows from the local one.

This paper is organized as follows.
In Section \ref{section2}, we recall the notion of random measure-theoretical entropy, introduce the notions of random measure-theoretical entropy and topological entropy for covers and give the relationship of them.
In Section \ref{section3}, we introduce another notion of random measure-theoretical entropy for covers and  prove a variational inequality of random entropy.
In Section \ref{section4}, we give some relations of the two notions of random measure-theoretical entropy for covers and prove the local variational principle for random bundle transformations.

\section{Preliminaries}\label{section2}
The setup consists of a probability space $(\Omega, \mathcal{F},P)$, together with a $P$-preserving transformation $\vartheta$, of a compact metric space $X$ together with the distance function $d$ and the Borel $\sigma$-algebra $\mathcal{B}$, and of a set $\mathcal{E}\subset \Omega\times X$ measurable with respect to the product $\sigma$-algebra $\mathcal{F}\times \mathcal{B}$ and such that the fibers $\mathcal{E}_{\omega}=\{x\in X: (\omega,x)\in \mathcal{E}\}$, $\omega\in \Omega$, are compact. The latter (see \cite{Castaing}) means that the mapping $\omega\rightarrow \mathcal{E}_{\omega}$ is measurable with respect to the Borel $\sigma$-algebra induced by the Hausdorff topology on the space $\mathcal{K}(X)$ of compact subsets of $X$ and that the distance function $d(x,\mathcal{E}_{\omega})$ is measurable in $\omega\in \Omega$ for each $x\in X$.  We assume that $\mathcal{F}$ is complete, countably generated, and separated points, and so $(\Omega, \mathcal{F},P)$ is a Lebesgue space.
A continuous (or homeomorphic) bundle
random dynamical system (RDS) $T$ over $(\Omega,\mathcal {F},
P, \vartheta)$ is generated by map $T_{\omega}:
\mathcal{E_{\omega}}\rightarrow \mathcal{E_{\vartheta \omega}}$ with
iterates $T_{\omega}^{n}=T_{\vartheta^{n-1}\omega} \cdots
T_{\vartheta\omega}T_{\omega}$, $ n \geq 1$,  and
$T_{\omega}^{0}=id$, so that the map $(\omega, x) \rightarrow
T_{\omega}x$ is measurable and the map $x \rightarrow T_{\omega}x$
is a continuous map (or a homeomorphism, respectively) for $P$-almost surely (a.s.) $\omega$. The map
$\Theta : \mathcal{E} \rightarrow \mathcal{E}$ defined by
$\Theta(\omega,x)=$ $(\vartheta\omega, T_{\omega}x)$ is called the skew
product transformation.

A {\it cover}  is a finite family of measurable subsets $\{U_i\}_{i=1}^k$ of $\Omega\times X$  such that $\bigcup_{i=1}^kU_i=\Omega\times X$ and  the $\omega$-section $U_i(\omega)=\{x\in X:(\omega,x)\in \Omega\times X\}$ is  Borel for each $i\in \{1,\dots, k\}$. Obviously, $\mathcal{U}(\omega)=\{U_i(\omega)\}_{i=1}^k$ is a Borel cover of $X$ in the usual sense.
 A {\it partition} of $\Omega\times X$   is a cover of $\Omega\times X$ whose elements are pairwise disjoint. An {\it open   cover } of $\Omega\times X$ is a  cover of $\Omega\times X$ such that  the $\omega$-sections of whose elements are open   subsets of $X$. Denote the sets of partitions by $\mathcal{P}_{\Omega\times X}$, the sets of covers by $\mathcal{C}_{\Omega\times X}$, the sets of open covers by $\mathcal{C}_{\Omega\times X}^o$.
Denote $\mathcal{C}_{\Omega\times X}^{o'}$ by the family of $\mathcal{U}\in \mathcal{C}_{\Omega\times X}^{o}$ with the form $\mathcal{U}=\{\Omega\times U_i\}$, where $\{U_i\}$ is an open cover of $X$. Clearly $\mathcal{C}_{\Omega\times X}^{o'}\subset \mathcal{C}_{\Omega\times X}^{o}$.
 For the measurable subset $\mathcal{E}\subset \Omega\times X$,
we denote the restriction of  $\mathcal{P}_{\Omega\times X}$, $\mathcal{C}_{\Omega\times X}$, $\mathcal{C}_{\Omega\times X}^o$ and $\mathcal{C}_{\Omega\times X}^{o'}$ on $\mathcal{E}$ by $\mathcal{P}_{\mathcal{E}}$, $\mathcal{C}_{\mathcal{E}}$, $\mathcal{C}_{\mathcal{E}}^o$ and $\mathcal{C}_{\mathcal{E}}^{o'}$, respectively. Given two covers $\mathcal{U}$, $\mathcal{V}\in \mathcal{C}_{\Omega\times X}$, $\mathcal{U}$ is said to be {\it finer} than $\mathcal{V}$ (write $\mathcal{U}\succeq \mathcal{V}$) if each element of $\mathcal{U}$ is contained in some element of $\mathcal{V}$. Let $\mathcal{U}\vee\mathcal{V}=\{U\cap V: U\in \mathcal{U}, V\in \mathcal{V}\}$. Given integers $M,$ $N\in \mathbb{N}$ with $M\leq N$ and $\mathcal{U}\in \mathcal{C}_{\Omega\times X}$ or $\mathcal{P}_{\Omega\times X}$, we use the notation $\mathcal{U}_M^N=\bigvee_{n=M}^N\Theta^{-n}\mathcal{U}$.

Given $\mathcal{U}\in \mathcal{C}_{\mathcal{E}}$, Let
\begin{equation*}
N(T,\omega,\mathcal{U},n)=\min\{\# F: F\,\, \text{is the finite subcover of }\,\, \bigvee_{i=0}^{n-1}(T_{\omega}^i)^{-1}\mathcal{U}(\vartheta^i\omega) \,\,  \text{over}\,\, \mathcal{E}_{\omega}    \},
\end{equation*}
where $\# F$ denotes the cardinality of $F$.
Note that $N(T,\omega,\mathcal{U},n)$ is measurable in $\omega$.
Then we can let
\begin{equation}
H(T,\mathcal{U},n)=\int \log N(T,\omega,\mathcal{U},n)dP(\omega).
\end{equation}
Clearly, if there is another cover $\mathcal{V}\succeq \mathcal{U}$ then $H(T,\mathcal{V},n)\geq H(T,\mathcal{U},n)$. In fact, for two covers $\mathcal{U}, \mathcal{V}$ we have $H(T,\mathcal{U}\vee\mathcal{V},n)\leq H(T,\mathcal{U},n)+H(T,\mathcal{V},n)$.

We now proceed by defining the random topological entropy of a cover $\mathcal{U}\in \mathcal{C}_{\mathcal{E}}$ on $\mathcal{E}$. Let $a_n=\log N(T,\omega,\mathcal{U},n)$. It is easy to see that $\{a_n\}$ is a non-negative subadditive sequence, i.e. $a_{n+m}\leq a_n+a_m\circ\vartheta^n$. Then by Kingman's subadditive ergodic theorem one can define the {\it random topological entropy of $\mathcal{U}$ on $\mathcal{E}$} as
\begin{equation*}
h_{\text{top}}(T,\mathcal{U})=\lim_{n\rightarrow\infty}\frac{1}{n}H(T,\mathcal{U},n)=\inf_{n\geq 1}\frac{1}{n}H(T,\mathcal{U},n).
\end{equation*}
Moreover, if $P$ is ergodic then $P$-$a.s.$
\begin{equation}
h_{\text{top}}(T,\mathcal{U})=\lim_{n\rightarrow\infty}\frac{1}{n}\log N(T,\omega,\mathcal{U},n).
\end{equation}
The {\it random topological entropy of $T$ on $\mathcal{E}$} is defined by
\begin{equation}
h_{\text{top}}(T)=\sup_{\mathcal{U}\in \mathcal{C}_{\mathcal{E}}^{o'}}h_{\text{top}}(T,\mathcal{U}).
\end{equation}

Note that the definition of $N(T,\omega,\mathcal{U},n)$
(hence $h_{\text{top}}(T,\mathcal{U})$) above is slightly different from the one of $\pi_T(f)(\omega,\epsilon,n)$ given in \cite{Kifer2001}, which is defined by separated sets. However,
it is easy to see that the random topological entropy $h_{\text{top}}(T)$ defined above is the
same as the random topological pressure for the null function (i.e. the random topological entropy) defined in \cite{Kifer2001}.
If $(\Omega,\mathcal{F},P,\vartheta)$ is a trivial system, the above definition is  the standard topological entropy in the deterministic case.

Let $\mathcal {P}_{P}(\mathcal{E}) = \{\mu \in \mathcal
{P}_P(\Omega \times X) : \mu(\mathcal{E})=1\}$, where
$\mathcal {P}_P(\Omega \times X)$ is the space of
probability measures on $\Omega \times X$ with the marginal
$P$ on $\Omega$. Any $\mu \in \mathcal
{P}_P(\mathcal{E})$ on $\mathcal{E}$ can be disintegrated
as $d \mu(\omega, x)=d \mu_{\omega}(x)d P(\omega)$ (See
\cite{Dudley}), where $\mu_{\omega}$ are regular conditional
probabilities with respect to the $\sigma$-algebra $\mathcal
{F}_{\mathcal{E}}$ formed by all sets $(A\times X)\cap \mathcal{E}$
with $A \in \mathcal{F}$. Let $\mathcal {R}=
\{\mathcal{R}_{i}\}$ be a finite measurable partition of
$\mathcal{E}$, and denote $\mathcal {R}(\omega)=
\{\mathcal{R}_{i}(\omega)\}$, where $\mathcal{R}_{i}(\omega)=\{x\in
\mathcal{E}_{\omega}: (\omega,x)\in \mathcal{R}_{i}\}$ is a
partition of $\mathcal{E}_{\omega}$. The conditional entropy of $\mathcal{R}$ given the $\sigma$-algebra $\mathcal{F}_{\mathcal{E}}$ is defined by
\begin{equation}\label{kifer2.1}
 H_{\mu}(\mathcal{R}\mid \mathcal{F}_{\mathcal{E}})=-\int\sum_{i}\mu(\mathcal{R}_i\mid \mathcal{F}_{\mathcal{E}})\log \mu(\mathcal{R}_i\mid \mathcal{F}_{\mathcal{E}})d\mu=\int H_{\mu_{\omega}}(\mathcal{R}(\omega))dP(\omega),
\end{equation}
where $H_{\mu_{\omega}}(\mathcal{A})$ denotes the usual entropy of a partition $\mathcal{A}$.
Let $\mathcal
{I}_P (\mathcal{E} )$ be the set of $\Theta$-invariant
measures $\mu \in\mathcal {P}_P(\mathcal{E})$. If $\vartheta$ is invertible, then $\mu$ is
$\Theta$-invariant if and only if the disintegrations $\mu_{\omega}$
of $\mu$ satisfy $T_{\omega}\mu_{\omega}=\mu_{\vartheta\omega}\,
P$-a.s. \cite{Arnold}. The {\it random measure-theoretical entropy $h_{\mu}^{(r)}(T)$  with respect to $\mu\in \mathcal{I}_P(\mathcal{E})$ } is defined by the formula
\begin{equation}\label{kifer2.2}
h_{\mu}^{(r)}(T)=\sup_{\mathcal{Q}}h_{\mu}^{(r)}(T,\mathcal{Q}), \,\,\text{where}\,\,
h_{\mu}^{(r)}(T,\mathcal{Q})=\lim_{n\rightarrow\infty}\frac{1}{n}
H_{\mu}(\bigvee_{i=0}^{n-1}(\Theta^i)^{-1}\mathcal{Q}\mid \mathcal{F}_{\mathcal{E}}),
\end{equation}
where the supremum is taken over all finite or countable measurable partitions $\mathcal{Q}=\{\mathcal{Q}_i\}$ of $\mathcal{E}$ with the finite conditional entropy $H_{\mu}(\mathcal{Q}\mid \mathcal{F}_{\mathcal{E}})<\infty$, and the above limit exists in view of subadditivity of conditional entropy (cf. Kifer \cite[Theorem
 II.1.1]{Kifer}).  From equality \eqref{kifer2.1}, it is not hard to see that $h_{\mu}^{(r)}(T,\mathcal{Q})$ has the following fiber expression:
\begin{equation}
h_{\mu}^{(r)}(T,\mathcal{Q})=\lim_{n\rightarrow\infty}\frac{1}{n}\int H_{\mu_{\omega}}(\bigvee_{i=0}^{n-1}(T^i_{\omega})^{-1}\mathcal{Q}(\vartheta^i\omega))dP(\omega).
\end{equation}
Moreover, the resulting entropy remains the same by taking
the supremum  only over partitions $\mathcal{Q}$ of $\mathcal{E}$
into sets $Q_i$ of the form $Q_i=(\Omega\times P_i)\cap\mathcal{E}$,
where $\mathcal{P}=\{P_i\}$ is a partition of $X$ into measurable
sets, so that $Q_i(\omega)=P_i\cap\mathcal{E}_{\omega}$ (See
\cite{Bogen,Bogenthesis,Kifer} for detail).

As in the topological case, when $(\Omega,\mathcal{F},P,\vartheta)$ is a trivial system, the above notion is  the standard metric entropy of $T$ with respect to $\mu$ in the deterministic system. For the classical entropy of measure-theoretical entropy see \cite{Par,Walters},  for the classical theory of topological entropy see \cite{DGS, Walters}, and, for the entropy theory of random dynamical system we refer to \cite{Arnold,Bogen,Kifer,LQ} and the references given there. The relation between the random measure-theoretical and topological entropy is stated as follows.

\begin{theorem}[Bogensch{\"u}tz \cite{Bogen} and Kifer \cite{Kifer2001}]\label{theorem2.1}
Let $T$ be a continuous bundle RDS on $\mathcal{E}$ over $\vartheta$, where $\vartheta $ is invertible. Then $h_{\text{top}}(T)=\sup_{\mu\in \mathcal{I}_P(\mathcal{E})} h_{\mu}^{(r)}(T) $
\end{theorem}

\begin{remark}
The classical variational principle follows from Theorem \ref{theorem2.1} by taking $(\Omega,\mathcal{F},P,\vartheta)$ to be the trivial system.
\end{remark}

Following the ideas of Romagnoli \cite{Rom2003} and Huang {\it et al.} \cite{Huang2006}, we define a new notion of  random measure-theoretical entropy for covers, which extends definition \eqref{kifer2.2} to covers. It allows us to give a local version (for a given open cover $\mathcal{U}$) of Theorem \ref{theorem2.1}.
Let $T$ be a homeomorphic bundle RDS on $\mathcal{E}$ over $\vartheta$ and $\mu\in \mathcal{P}_P(\mathcal{E})$.
For $\mathcal{U}\in \mathcal{C}_{\mathcal{E}}$, let
\begin{equation}
H_{\mu}(\mathcal{U}\mid \mathcal{F}_{\mathcal{E}})=\inf_{\mathcal{R}\succeq \mathcal{U}, \mathcal{R}\in \mathcal{P}_{\mathcal{E}}}H_{\mu}(\mathcal{R}\mid \mathcal{F}_{\mathcal{E}}).
\end{equation}
It is not hard to prove that many of the properties of the conditional function $H_{\mu}(\mathcal{R}\mid \mathcal{F}_{\mathcal{E}})$ can be extended to $H_{\mu}(\mathcal{U}\mid \mathcal{F}_{\mathcal{E}})$ from partitions to covers; for details see \cite{Rom2003}.

We need the following basic result to prove Lemma \ref{lemma2.2}. For $\{p_k\}_{k=1}^K\subset(0,1)$ with $\sum_{k=1}^Kp_k\leq 1$, as usual we define $H(p_1,\dots,p_K)=-\sum_{k=1}^Kp_k\log p_k$.
\begin{lemma}
\label{lemma7}
Fix $K\geq 2$. Suppose that $p_1,\dots,p_K\in (0,1)$ with $p_1\leq\dots\leq p_K$ and $\sum_{k=1}^Kp_k\leq 1$. Let $0< \delta_1<p_1$ and suppose that for each $k=2,\dots,K$, $\delta_k\in[0,1-p_k)$ and $\sum_{k=2}^K\delta_k=\delta_1$. Then
$$
H(p_1,\dots,p_K)>H(p_1-\delta_1,p_2+\delta_2,\dots,p_K+\delta_K).
$$
\end{lemma}

\begin{lemma}
\label{lemma2.2}
Let $T$ be a homeomorphic bundle RDS on $\mathcal{E}$ over $\vartheta$ and $\mu\in \mathcal{P}_P(\mathcal{E})$. If $\mathcal{U}$, $\mathcal{V}\in \mathcal{C}_{\mathcal{E}}$, then
\renewcommand{\theenumi}{(\arabic{enumi})}
\begin{enumerate}
\item $0\leq H_{\mu}(\mathcal{U}\mid \mathcal{F}_{\mathcal{E}})\leq \log N(\mathcal{U})$,  where $N(\mathcal{U})=\min\{\#F: \mathcal{F}\subset\mathcal{U}, \,\bigcup_{F\in \mathcal{F}}\supset \mathcal{E}\}$;
\item If $\mathcal{U}\succeq \mathcal{V}$, then  $H_{\mu}(\mathcal{U}\mid \mathcal{F}_{\mathcal{E}})\geq H_{\mu}(\mathcal{V}\mid \mathcal{F}_{\mathcal{E}})$;
\item $H_{\mu}(\mathcal{U}\vee \mathcal{V}\mid \mathcal{F}_{\mathcal{E}})\leq H_{\mu}(\mathcal{U}\mid \mathcal{F}_{\mathcal{E}})+H_{\mu}(\mathcal{V}\mid \mathcal{F}_{\mathcal{E}})$;\label{lemma2.2p3}
\item If $\vartheta$ is invertible, then $H_{\mu}(\Theta^{-1}\mathcal{U}\mid \mathcal{F}_{\mathcal{E}} )=H_{\Theta\mu}(\mathcal{U}\mid \mathcal{F}_{\mathcal{E}} )$.\label{lemma2.2p4}
\end{enumerate}
\end{lemma}

\begin{proof}
Part (1), (2) and (3) are obvious.

We now prove part (4). We follow the arguments applied in  \cite{Rom2003}.
Since
\begin{align*}
H_{\mu}(\Theta^{-1}\mathcal{U}\mid \mathcal{F}_{\mathcal{E}})
&
\leq \inf_{\mathcal{R}\in \mathcal{P}_{\mathcal{E}}, \mathcal{R}\succeq \mathcal{U}\cap \mathcal{E}}H_{\mu}(\Theta^{-1}\mathcal{R}\mid \mathcal{F}_{\mathcal{E}})=\inf_{\mathcal{R}\in \mathcal{P}_{\mathcal{E}}, \mathcal{R}\succeq \mathcal{U}\cap \mathcal{E}}H_{\Theta\mu}(\mathcal{R}\mid \mathcal{F}_{\mathcal{E}}),
\end{align*}
then $H_{\mu}(\Theta^{-1}\mathcal{U}\mid \mathcal{F}_{\mathcal{E}})\leq H_{\Theta\mu}(\mathcal{U}\mid \mathcal{F}_{\mathcal{E}} )$.

We now prove the opposite inequality.
Let $\mathcal{U}\in \mathcal{C}_{\mathcal{E}}$. For each $\mathcal{R}\in \mathcal{P}_{\mathcal{E}}$ with $\mathcal{R}\succeq \Theta^{-1}\mathcal{U}$, we recursively construct a $\mathcal{Q}\succeq \mathcal{U}$ such that $H_{\mu}(\mathcal{R}\mid \mathcal{F}_{\mathcal{E}})\geq  H_{\Theta\mu}(\mathcal{Q}\mid \mathcal{F}_{\mathcal{E}} )$.

Let $\mathcal{U}=\{U_1,\dots,U_M\}$ and $\mathcal{R}=\{R_1,\dots,R_K\}\succeq \Theta^{-1}\mathcal{U}$
with $\emptyset \neq R_k\subseteq \Theta^{-1}U_{j_k}$, $j_k\in\{1,\dots,M\}$, for each $k\in\{1,\dots,K\}$. We may assume that $j_k\neq j_l$ if $k\neq l$, since the partition obtained by replacing $R_k\cup R_l$ is coarser than $\mathcal{R}$ and still finer than $\Theta^{-1}\mathcal{U}$. Notice that $\Theta^{-1}U_{j_1}\backslash \bigcup_{l=2}^K\Theta^{-1}U_{j_l}\subseteq R_1$. For each $k=1,\dots,K$ we define $p_k=\mu(R_k\mid \mathcal{F}_{\mathcal{E}})$ and $p_k(\omega)=\mu_{\omega}(R_k(\omega))$. Then
$$\sum_{k=1}^Kp_k=\sum_{k=1}^K \mu(R_k\mid \mathcal{F}_{\mathcal{E}})= \sum_{k=1}^Kp_k(\omega)=\mu_{\omega}(R_k(\omega))=1,\,\, P\text{-a.s.}\,\, \omega,$$
and
$
H_{\mu_{\omega}}(\mathcal{R}(\omega))= H(p_1(\omega ),\dots, p_K(\omega)) \,\, P\text{-a.s.}\,\, \omega.
$
 By exchanging indices if necessary we assume $p_1\leq p_2\leq \dots\leq p_k$. Let us define
\begin{align*}
\delta_1&=\mu(R_1\mid\mathcal{F}_{\mathcal{E}})-\Theta\mu(U_{j_1}\backslash\bigcup_{l=2}^K U_{j_l}\mid \mathcal{F}_{\mathcal{E}})\\
&=\mu(R_1\mid\mathcal{F}_{\mathcal{E}})-\mu(\Theta^{-1}(U_{j_1}\backslash\bigcup_{l=2}^K  U_{j_l})\mid \Theta^{-1}(\mathcal{F}_{\mathcal{E}}))
\end{align*}
Then for $P$-a.s. $\omega$,
\begin{equation*}
\delta_1(\omega)=\mu_{\omega}(R_1(\omega))-T_{\omega}\mu_{\omega}(U_{j_1}(\vartheta\omega)\backslash \bigcup_{l=2}^KU_{j_l}(\vartheta\omega))\geq 0.
\end{equation*}

Define $B^1_1=U_{j_1}\backslash \bigcup_{l=2}^KU_{j_l}$ and $R_2^1=R_2\cup (\Theta^{-1}U_{j_2}\cap (R_1\backslash \Theta^{-1}B_1^1))$. For $k=3,\dots,K$, let
\begin{equation*}
R_k^1=R_k\cup\big( \Theta^{-1}U_{j_k}\cap (R_1\backslash \Theta^{-1}B_1^1)\cap (\bigcup_{l=2}^{k-1}R_l^1)^c  \big).
\end{equation*}
Then for each $\omega\in\Omega$,
$B^1_1(\vartheta\omega)=U_{j_1}(\vartheta\omega)\backslash \bigcup_{l=2}^KU_{j_l}(\vartheta\omega)$ and $R_2^1(\omega)=R_2(\omega)\cup (T_{\omega}^{-1}U_{j_2}(\vartheta\omega)\cap (R_1(\omega)\backslash T_{\omega}^{-1}B_1^1(\vartheta\omega)))$. For $k=3,\dots,K$,
\begin{equation*}
R_k^1(\omega)=R_k(\omega)\cup\big( T_{\omega}^{-1}U_{j_k}(\vartheta\omega)\cap (R_1(\omega)\backslash T_{\omega}^{-1}B_1^1(\vartheta\omega))\cap (\bigcup_{l=2}^{k-1}R_l^1(\omega))^c  \big),
\end{equation*}
where $(\bigcup_{l=2}^{k-1}R_l^1(\omega))^c=\mathcal{E}_{\omega}\backslash \bigcup_{l=2}^{k-1}R_l^1(\omega)$.

Define $\mathcal{R}_1=\{\Theta^{-1}B_1^1,R_2^1,\dots,R_K^1\}$. It is clear that $\mathcal{R}_1\succeq \Theta^{-1}\mathcal{U}$. If $\delta_1=0$ then for $P$-a.s. $\omega\in\Omega$, $H_{\mu_{\omega}}(\mathcal{R}(\omega))=H_{\mu_{\omega}}(\mathcal{R}_1(\omega))$, it follows that $H_{\mu}(\mathcal{R}\mid \mathcal{F}_{\mathcal{E}})=H_{\mu}(\mathcal{R}_1\mid \mathcal{F}_{\mathcal{E}})$. If $\delta_1>0$, then for $P$-a.s. $\omega\in\Omega$, $\delta_1(\omega)>0$, Using Lemma \ref{lemma7} with $(p_1(\omega),\dots,p_K(\omega))$ and $\delta_k(\omega)=\mu_{\omega}(R_k^1(\omega))-\mu_{\omega}(R_k(\omega))$ for every $k\in\{2,\dots,K\}$, we have $H_{\mu_{\omega}}(\mathcal{R}(\omega))\geq H_{\mu_{\omega}}(\mathcal{R}_1(\omega))$, $P$-a.s. $\omega\in \Omega$. Then $H_{\mu}(\mathcal{R}\mid \mathcal{F}_{\mathcal{E}})\geq H_{\mu}(\mathcal{R}_1\mid \mathcal{F}_{\mathcal{E}})$.

Inductively, for each $n=2,\dots,K$, we construct
$$
\mathcal{R}_n=\{\Theta^{-1}B_1^n,\dots, \Theta^{-1}B_n^n,R_{n+1}^n,\dots,R_K^n\}\succeq \Theta^{-1}\mathcal{U},
$$
which satisfies that $H_{\mu}(\mathcal{R}_{n+1}\mid \mathcal{F}_{\mathcal{E}})\leq H_{\mu}(\mathcal{R}_n\mid \mathcal{F}_{\mathcal{E}})$. Let us give the construction of $\mathcal{R}_{n+1}$ given $\mathcal{R}_n$.
For each $k\in \{1,\dots,n\}$ define $B_k^{n+1}=B_k^n$. Let $B_{n+1}^{n+1}=U_{j_{n+1}}\backslash \bigcup_{l=n+2}^KU_{j_l}$ and $R_{n+2}^{n+1}=R_{n+2}^n\cup (\Theta^{-1}U_{j_{n+2}}\cap (R_{n+1}^n\backslash \Theta^{-1}B_{n+1}^{n+1}))$. For each $k\in \{n+3,\dots,K\}$, define
$$
R_k^{n+1}=R_k^n\cap \big(\Theta^{-1}U_{j_k}\cap (R_{n+1}^n\backslash \Theta^{-1}B_{n+1}^{n+1})\cap (\bigcup_{l=2}^{k-1}R_l^{n+1})^c    \big).
$$
As before, for each $k\in \{1,\dots,k-n\}$ define $p_k=\mu(R_{n+k}^n\mid\mathcal{F}_{\mathcal{E}})$ and $\delta_k=\mu(R_{n+k}^{n+1}\mid\mathcal{F}_{\mathcal{E}})-\Theta\mu (R_{n+k}^n\mid\mathcal{F}_{\mathcal{E}}))$. Let $\delta_1=\mu(R_{n+1}^n\mid\mathcal{F}_{\mathcal{E}}))
-\Theta\mu(B_{n+1}^{n+1})\mid\mathcal{F}_{\mathcal{E}})$. By exchanging indices if necessary we assume that $p_1\leq \dots\leq p_{K-n}$. Using Lemma \ref{lemma7}, we prove that $H_{\mu}(\mathcal{R}_{n+1}\mid\mathcal{F}_{\mathcal{E}})\leq H_{\mu}(\mathcal{R}_n)\mid\mathcal{F}_{\mathcal{E}})$.

We have that $H_{\mu}(\mathcal{R}\mid\mathcal{F}_{\mathcal{E}})\geq H_{\mu}(\mathcal{R}_K\mid\mathcal{F}_{\mathcal{E}})$ and $\mathcal{R}_K=\Theta^{-1}\mathcal{Q}$ with $\mathcal{Q}\succeq\mathcal{U}$. Then
\begin{align*}
H_{\mu}(\Theta^{-1}\mathcal{U}\mid\mathcal{F}_{\mathcal{E}})&=\inf_{\mathcal{R}\succeq \Theta^{-1}\mathcal{U}}H_{\mu}(\mathcal{R}\mid\mathcal{F}_{\mathcal{E}})\\
&\geq \inf_{\mathcal{Q}\succeq \mathcal{U}}H_{\mu}(\Theta^{-1}\mathcal{Q}\mid\mathcal{F}_{\mathcal{E}})\\
&=\inf_{\mathcal{Q}\succeq \mathcal{U}}H_{\Theta\mu}(\mathcal{Q}\mid\mathcal{F}_{\mathcal{E}}) \quad (\text{since }\,\, \vartheta\,\, \text{is invertible})\\
&=H_{\Theta\mu}(\mathcal{U}\mid\mathcal{F}_{\mathcal{E}}).
\end{align*}
This complete the argument of Lemma \ref{lemma2.2}.

\end{proof}

Let $\mu\in \mathcal{I}_P(\mathcal{E})$. It follows from parts \ref{lemma2.2p3} and \ref{lemma2.2p4} of Lemma \ref{lemma2.2} that $H_{\mu}(\mathcal{U}_0^{n-1}\mid \mathcal{F}_{\mathcal{E}})$ is a sub-additive function of $n\in \mathbb{N}$.
Let $\vartheta$ be invertible. We  may define the $\mu^-$ random measurable theoretic entropy of $\mathcal{U}$ with respect to $\mathcal{F}_{\mathcal{E}}$ as
\begin{equation*}
h_{\mu}^{(r)-}(T,\mathcal{U})=\lim_{n\rightarrow \infty}\frac{1}{n}H_{\mu}(\mathcal{U}_0^{n-1}\mid \mathcal{F}_{\mathcal{E}})=\inf_{n\geq 1}\frac{1}{n}H_{\mu}(\mathcal{U}_0^{n-1}\mid \mathcal{F}_{\mathcal{E}}).
\end{equation*}
Since for each partition $\mathcal{U}\in \mathcal{P}_{\mathcal{E}}$, $h_{\mu}^{(r)-}(T,\mathcal{U})=h_{\mu}^{(r)}(T,\mathcal{U})$, then $h_{\mu}^{(r)}(T)=\sup_{\mathcal{U}\in \mathcal{C}_{\mathcal{E}}}h_{\mu}^{(r)-}(T,\mathcal{U})$.
The following lemma gives some stronger results.

\begin{lemma}
\label{lemma2.3}
Let $T$ be a homeomorphic bundle RDS on $\mathcal{E}$ over $\vartheta$ and $\mu \in \mathcal{I}_P(\mathcal{E})$, where $\vartheta $ is invertible. Then
\renewcommand{\theenumi}{(\arabic{enumi})}
\begin{enumerate}
\item $h_{\mu}^{(r)}(T)=\sup_{\mathcal{U}\in \mathcal{C}_{\mathcal{E}}^o}h_{\mu}^{(r)-}(T,\mathcal{U})$;
\item $h_{\mu}^{(r)-}(T,\mathcal{U})\leq h_{\text{top}}(T,\mathcal{U})$ for each $\mathcal{U}\in \mathcal{C}_{\mathcal{E}}$.
\end{enumerate}
\end{lemma}

\begin{proof}
We follow the idea of Huang Wen {\it et al.} \cite{Huang2006}.

(1) For $\mathcal{U}\in \mathcal{C}_{\mathcal{E}}^o$, let $\mathcal{R}$ be the partition generated by $\mathcal{U}$. Then $\mathcal{R}\succeq \mathcal{U}$, and hence $h_{\mu}^{(r)}(T)\geq h_{\mu}^{(r)}(T,\mathcal{R})\geq h_{\mu}^{(r)-}(T,\mathcal{U})$. This implies that $h_{\mu}^{(r)}(T)\geq \sup_{\mathcal{U}\in \mathcal{C}_{\mathcal{E}}^o}h_{\mu}^{(r)-}(T,\mathcal{U})$.

Conversely, for a partition $\mathcal{R}\in \mathcal{P}_{\mathcal{E}}$, $\mathcal{R}=\{R_i\}_{i=1}^k$
with $R_i=(\Omega \times P_i)\cap \mathcal{E}$, where $P=\{P_i\}$ is the partition of $X$ into measurable sets such that $R_i(\omega)=P_i\cap \mathcal{E}_{\omega}$, and $\epsilon >0$, we have the following claim.
\begin{claim}
There exists an open cover $\mathcal{U}\in \mathcal{C}_{\mathcal{E}}^o$ with $K$ elements such that for any finite measurable partition $\mathcal{Q}\in \mathcal{P}_{\mathcal{E}}$ with $\mathcal{Q}=(\Omega \times Q_i)\cap \mathcal{E}$, $\mathcal{Q}_i(\omega)=Q_i\cap \mathcal{E}_{\omega}$, finer than $\mathcal{U}$ as a cover, $H_{\mu}(\mathcal{R}\mid \mathcal{Q}\vee \mathcal{F}_{\mathcal{E}})<\epsilon$.
\end{claim}

\begin{proof}[Proof the Claim]
Let $\mathcal{P}=\{P_i\}_{i=1}^k$ be a finite partition of $ X$. Denote by $\mathcal{P}(\omega)=\{P_i(\omega)\}_{i=1}^k$, $P_i(\omega)=P_i\cap \mathcal{E}_{\omega}$, $1\leq i\leq k$, the corresponding partition of $\mathcal{E}_{\omega}$. It is well known that  there exists $\delta(\omega)>0$ such that if $\beta=\{B_i\}_{i=1}^k$ is a measurable partition of $X$ and $\sum_{i=1}^k\mu_{\omega}(P_i\bigtriangleup B_i)<\delta(\omega)$  then $H_{\mu_{\omega}}(\mathcal{P}\mid \beta )\leq \epsilon$ (See \cite{Walters}). Since $\mu_{\omega}$ is regular, we can find compact subsets $Q_i\subset P_i$ with
$$
\mu(\Omega\times P_i\backslash  \Omega \times Q_i)=\int \mu_{\omega}(P_i(\omega)\backslash Q_i(\omega))dP(\omega)<\delta/2k^2, i=1, \dots, k,
$$
where $\delta=\int\delta(\omega)dP(\omega)$.
Let $Q_0=X\backslash \bigcup_{i=1}^k Q_i$. Then $\mu(\Omega\times Q_0)<\delta/2k$. Let $U_i=\Omega \times B_i$, where $B_i=Q_0\cup Q_i$, $i=1,\dots, k$, is open in $X$. Then $\mathcal{U}=\{U_i\}_{i=1}^k\cap \mathcal{E}\in \mathcal{C}_{\mathcal{E}}^o$.

For any partition $\mathcal{S}\succeq \mathcal{U}$, $\mathcal{S}=\{S_i\}\in \mathcal{P}_{\mathcal{E}}$ with $S_i=(\Omega\times C_i)\cap \mathcal{E}$, where $\mathcal{C}=\{C_i\}$ is a partition of $X$, we can find a partition $\mathcal{S}'=\{S'_i\}_{i=1}^k$ satisfying that $C'_i\subset B_i$, $S'_i\subset U_i$, $i=1,\dots, k$ and $\mathcal{S}\succeq \mathcal{S}'$, where $S_i'=(\Omega\times C_i')\cap\mathcal{E}$. Hence $H_{\mu}(\mathcal{R}\mid \mathcal{S}\vee \mathcal{F}_{\mathcal{E}})\leq H_{\mu}(\mathcal{R}\mid \mathcal{S}'\vee \mathcal{F}_{\mathcal{E}})$. Note that $B_i\supset C'_i\supset X\backslash \bigcup_{j\neq i}B_j=Q_i$, and thus $U_i\supset S'_i\supset (\Omega\times X)\backslash \bigcup_{j\neq i}(\Omega \times B_i)=\Omega\times Q_i$. One has
$$
\mu_{\omega}(C'_i\bigtriangleup P_i)\leq \mu_{\omega}(P_i\backslash Q_i)+\mu_{\omega}(Q_0)\leq \delta(\omega)/2k+\delta(\omega)/2k=\delta(\omega)/k.
$$
Hence $\sum_{i=1}^k\mu_{\omega}(C'_i\bigtriangleup P_i)<\delta(\omega)$ and $H_{\mu_{\omega}}(\mathcal{R}(\omega)\mid \mathcal{S}'(\omega))\leq \epsilon$. Then
\begin{align*}
H_{\mu}(\mathcal{R}\mid \mathcal{S}'\vee \mathcal{F}_{\mathcal{E}})
&=H_{\mu}(\mathcal{R}\vee \mathcal{S}'\mid \mathcal{F}_{\mathcal{E}})-H_{\mu}(\mathcal{S}'\mid \mathcal{F}_{\mathcal{E}})\\
&=\int H_{\mu_{\omega}}(\mathcal{R}\vee\mathcal{S}')(\omega)dP(\omega)-\int H_{\mu_{\omega}}(\mathcal{S}'(\omega))dP(\omega)\\
&=\int (H_{\mu_{\omega}}(\mathcal{R}\vee\mathcal{S}')(\omega)- H_{\mu_{\omega}}(\mathcal{S}'(\omega)))dP(\omega)\\
&=\int H_{\mu_{\omega}}(\mathcal{R}(\omega)\mid \mathcal{S}'(\omega))dP(\omega)\leq \epsilon.
\end{align*}
Thus $H_{\mu}(\mathcal{R}\mid \mathcal{S}\vee \mathcal{F}_{\mathcal{E}})<\epsilon$. This ends the proof the claim.
\end{proof}

Now for $n\in \mathbb{N}$ and a finite measurable partition $\mathcal{Q}_n\succeq \mathcal{U}_{0}^{n-1}$, since $\Theta^i\mathcal{Q}_n\succeq \mathcal{U}$, $0\leq i\leq n-1$, one has
\begin{align*}
H_{\mu}(\mathcal{R}_0^{n-1}\mid \mathcal{F}_{\mathcal{E}})
&\leq H_{\mu}(\mathcal{Q}_n\mid \mathcal{F}_{\mathcal{E}})+H_{\mu}(\mathcal{R}_0^{n-1}\mid \mathcal{Q}_n\vee \mathcal{F}_{\mathcal{E}})\\
& \leq H_{\mu}(\mathcal{Q}_n\mid \mathcal{F}_{\mathcal{E}})+\sum_{i=0}^{n-1}H_{\mu}(\mathcal{R}\mid \Theta^i\mathcal{Q}_n\vee \mathcal{F}_{\mathcal{E}})\\
& \leq H_{\mu}(\mathcal{Q}_n\mid \mathcal{F}_{\mathcal{E}})+n\epsilon.
\end{align*}

Hence
\begin{align*}
h_{\mu}^{(r)}(T,\mathcal{R})
&=\lim_{n\rightarrow \infty}\frac{1}{n}H_{\mu}(\mathcal{R}_0^{n-1}\mid \mathcal{F}_{\mathcal{E}})\\
&\leq \lim_{n\rightarrow \infty}\frac{1}{n}H_{\mu}(\mathcal{U}_0^{n-1}\mid \mathcal{F}_{\mathcal{E}})+\epsilon\\
&=h_{\mu}^{(r)-}(T,\mathcal{U})+\epsilon
\leq \sup_{\mathcal{U}\in \mathcal{C}^o_{\mathcal{E}}}h_{\mu}^{(r)-}(T,\mathcal{U})+\epsilon.
\end{align*}
Since $\mathcal{R}$ and $\epsilon$ are arbitrary, one has $h_{\mu}^{(r)}(T)\leq \sup_{\mathcal{U}\in \mathcal{C}^o_{\mathcal{E}}}h_{\mu}^{(r)-}(T,\mathcal{U})$. This ends the proof part (1).

(2) It is enough to show that for each $\mathcal{U}=\{U_i\}_{i=1}^k\in \mathcal{C}_{\mathcal{E}}$ there exists $\mathcal{R}\in \mathcal{P}_{\mathcal{E}}$ finer than $\mathcal{U}$ such that  $H_{\mu}(\mathcal{R}\mid \mathcal{F}_{\mathcal{E}})\leq H(T,\mathcal{U},1)$ which implies  $H_{\mu}(\mathcal{U}\mid \mathcal{F}_{\mathcal{E}})\leq H(T,\mathcal{U},1)$.

Let $d\mu(\omega,x)=d\mu_{\omega}(x)dP(\omega)$ $P$-a.s. be the integration of $\mu$. For each $\omega\in \Omega$, there exists $I_{\omega}\subset \{1,\dots,k\}$ with cardinality $N(T,\omega,\mathcal{U},1)$ such that $\bigcup_{i\in I_{\omega}}U_i(\omega)\supset \mathcal{E}_{\omega}$. Since $\mathcal{U}$ is finite, we can find $\omega_1,\dots,\omega_s\in \Omega$ such that for each $\omega\in\Omega$, $I_{\omega}=I_{\omega_i}$ for some $i\in \{1,\dots,s\}$. For $i=1,\dots,s$, define $D_i=\{\omega\in \Omega: \mu_{\omega}(\bigcup_{j\in I_{\omega_i}}U_j(\omega))=1\}$. Then $D_i$ is measurable for $i=1,\dots,s$ and $P(\bigcup_{i=1}^sD_i)=1$.

Fix $i\in\{1,\dots,s\}$. Assume that $I_{\omega_i}=\{k_1<\cdots<k_{t_i}\}$, where $t_i=N(T,\omega_i,\mathcal{U},1)$. Take $\mathcal{R}^i=\{W^i_1,\dots, W_{t_i}^i\}$, where $W_1^i=U_{k_1}$, $W^i_2=U_{k_2}\backslash U_{k_1}$, $\cdots$, $W_{t_i}^i=U_{k_{t_i}}\backslash \bigcup_{j=1}^{t_i-1}U_{k_j}$. For any $\omega\in D_i$, since $\mu_{\omega}(\bigcup_{j=1}^{t_i}W_j^i(\omega_i))=\mu_{\omega}(\bigcup_{j\in I_{\omega_i}}U_j(\omega))=1$. $\mathcal{R}^i(\omega_i)$ can be considered as  a finite partition of $\mathcal{E}_{\omega}$ (mod $\mu_{\omega}$) and
$$
H_{\mu_{\omega}}(\mathcal{R}^i(\omega_i))\leq \log N(T,\omega_i,\mathcal{U},1).
$$

Let $C_1=D_1$, $C_i=D_i\backslash \bigcup_{j=1}^{i-1}D_j$, $i=1,\dots,s$ and $A=\mathcal{E}\backslash (\bigcup_{i=1}^s(E_i\cap \bigcup_{j=1}^{t_i}W_j^i))$, where $E_i=\{(\omega,x)\in \mathcal{E}: \omega\in C_i, x\in \mathcal{E}_{\omega}\}$. Set $A_l=A\cap (U_l\backslash \bigcup_{j=1}^{l-1}U_j)$, $l=1,\dots, k$. Then put
$$
\mathcal{R}=\{E_1\cap W_1^1, \dots, E_1\cap W_{t_1}^1, \dots, E_s\cap W_1^s, E_s\cap W_{t_s}^s, A_1\cap\mathcal{E}, \dots, A_k\cap \mathcal{E}\}.
$$
Clearly, $\mathcal{R}$ is a finite measurable partition of $\mathcal{E}$ finer than $\mathcal{U}$. Now for $\omega\in C_i$, $i=1, \dots, s$, we have
$H_{\mu_{\omega}}(\mathcal{R}(\omega))=H_{\mu_{\omega}}(\mathcal{R}^i(\omega_i))$. Hence
\begin{align*}
H_{\mu}(&\mathcal{R}\mid \mathcal{F}_{\mathcal{E}})=\int H_{\mu_{\omega}}(\mathcal{R}(\omega))dP(\omega)\\
&=\int_{\bigcup_{i=1}^sC_i}H_{\mu_{\omega}}(\mathcal{R}(\omega))dP(\omega)\leq \int \log N(T,\omega,\mathcal{U},1)dP(\omega)=H(T,\mathcal{U},1).
\end{align*}
This ends the proof of Lemma \ref{lemma2.3}.

\end{proof}

\begin{remark}
The constructed $\mathcal{U}$ in fact belongs to $\mathcal{C}_{\mathcal{E}}^{o'}$ in the proof of Lemma \ref{lemma2.3} part (1). Then, more precisely, $h_{\mu}^{(r)}(T)=\sup_{\mathcal{C}_{\mathcal{E}}^{o'}}h_{\mu}^{(r)-}(T,\mathcal{U})=\sup_{\mathcal{C}_{\mathcal{E}}^{o}}h_{\mu}^{(r)-}(T,\mathcal{U})$.
When $(\Omega,\mathcal{F},P,\vartheta)$ is a trivial system, the inequality $h_{\mu}^{(r)-}(T,\mathcal{U})\leq h_{\text{top}}(T,\mathcal{U})$ can be easily obtained by the fact $H_{\mu}(\alpha)\leq \log \# \alpha$ for $\alpha\in \mathcal{P}_X$, where $\mathcal{P}_X$ is the set of partitions of $X$.
\end{remark}

Let $\mathcal{P}_X$ be the set of  partitions of $X$ and $\mathcal{M}(X)$ be the set of all Borel probability measures on $X$.
For any $\alpha\in \mathcal{P}_X$ and $\theta\in \mathcal{M}(X)$,
we define $|\mathcal{A}|_{\theta}=\# \{A\in \alpha: \theta(A)>0\}$. Then in the proof of Lemma \ref{lemma2.3} part (2), in fact we have obtained the following fact.

\begin{corollary}
Let $\mu\in \mathcal{P}_P(\mathcal{E})$ and $d\mu(\omega,x)=d\mu_{\omega}(x)dP(\omega)$ P-a.s. be the disintegration of $\mu$. Then for any $\mathcal{U}\in \mathcal{C}_{\mathcal{E}}$, there exists $\mathcal{R}\in \mathcal{P}_{\mathcal{E}}$ such that $\mathcal{R}\succeq \mathcal{U}$ and $\mid \mathcal{R}(\omega)\mid_{\mu_{\omega}}\leq  \sup_{\omega\in\Omega}N(T,\omega,\mathcal{U},1)$ for $P$-a.s. $\omega\in \Omega$.
\end{corollary}

It follows from Lemma \ref{lemma2.3} that the inequality stated in Theorem \ref{theorem2.1} holds true. In fact,  we have the following Theorem \ref{theorem2.5}. The proof of this result will be completed in Section \ref{section4}. In next section, we will introduce another new notion of random measure-theoretical entropy for covers, and prove a variational inequality for this new entropy. Using the inequality we can prove Theorem \ref{theorem2.5} in Section \ref{section4}.

\begin{theorem}
\label{theorem2.5}
Let $T$ be a homeomophic bundle RDS on $\mathcal{E}$ over $\vartheta$. Then for each $\mathcal{U}\in \mathcal{C}_{\mathcal{E}}^{o'}$ there exists $\mu\in \mathcal{I}_P(\mathcal{E})$ such that $h_{\text{top}}(T,\mathcal{U})=h_{\mu}^{(r)-}(T,\mathcal{U})$.
\end{theorem}

Theorem \ref{theorem2.5} together with Lemma \ref{lemma2.3} implies that $h_{\text{top}}(T,\mathcal{U})\leq  \sup_{\nu}h_{\nu}^{(r)}(T)$. By taking the supremum over all covers $\mathcal{U}\in \mathcal{C}_{\mathcal{E}}^{o'}$ in Theorem \ref{theorem2.5}, we can easily get Theorem \ref{theorem2.1} in the homeomorphic case. Moreover, Theorem \ref{theorem2.5} also shows that if there exists $\mathcal{U}\in \mathcal{C}_{\mathcal{E}}^{o'}$ such that $h_{\text{top}}(T,\mathcal{U})=h_{\text{top}}(T)$, then there exists $\mu\in \mathcal{I}_P(\mathcal{E})$ such that $h_{\mu}^{(r)}(T)=h_{\text{top}}(T)$.

\section{A variational inequality of random entropy for $h_{\mu}^{(r)+}$}\label{section3}

Let $T$ be a homeomorphic bundle RDS on $\mathcal{E}$ over $\vartheta$. Given $\mu\in \mathcal{I}_P(\mathcal{E})$ and $\mathcal{U}\in \mathcal{C}_{\mathcal{E}}$ we define
\begin{equation*}
h_{\mu}^{(r)+}(T,\mathcal{U})=\inf_{\mathcal{Q}\in \mathcal{P}_{\mathcal{E}},\mathcal{Q}\succeq \mathcal{U}}h_{\mu}^{(r)}(T,\mathcal{Q}).
\end{equation*}
Obviously, $h_{\mu}^{(r)+}(T,\mathcal{U})\geq h_{\mu}^{(r)-}(T,\mathcal{U})$. By Lemma \ref{lemma2.3} part (1),  $h_{\mu}^{(r)}(T)=\sup_{\mathcal{U}\in \mathcal{C}_{\mathcal{E}}^o}h_{\mu}^{(r)+}(T,\mathcal{U})$ also holds.
For $\mathcal{U}\in \mathcal{C}_{\mathcal{E}}^{o'} $,
it is not difficult to verify (See e.g. \cite{Bogen,Kifer}) that the infimum above can only over partitions $\mathcal{Q}$ of $\mathcal{E}$ into sets $Q_i$ of the form $Q_i=(\Omega\times P_i)\cap \mathcal{E}$, where $\mathcal{P}=\{P_i\}$ is a partition of $X$ into measurable subsets, so that $Q_i(\omega)=P_i\cap \mathcal{E}_{\omega}$.
In this section, we will show that, for  given $\mathcal{U}\in \mathcal{C}_{\mathcal{E}}^{o'} $, there always exists $\mu \in \mathcal{I}_P(\mathcal{E})$ such that $h_{\mu}^{(r)+}(T,\mathcal{U})\geq h_{\text{top}}(T,\mathcal{U})$.

First we recall the definition of factor for two RDS (See \cite{Liupeidong2005}).
\begin{definition}
 Given two continuous bundle RDS $T_{i}$ on $\mathcal{E}_i\subset \Omega \times X_i$ over $\vartheta$, $i=1,2$.
 $T_2$ is called a factor of $T_1$ if there exists a family of subjective continuous maps $\{\pi_{\omega}:(\mathcal{E}_{1})_{\omega}\rightarrow (\mathcal{E}_{2})_{\omega}\}$ such that for $P$-a.s. $\omega$, $ (T_{2})_{\omega}\circ \pi_{\omega}=\pi_{\vartheta\omega}\circ (T_{1})_{\omega}$ and $\pi:(\omega,x)\rightarrow (\omega, \pi_{\omega}x)$ constitutes a measurable map from $\mathcal{E}_1$ to $\mathcal{E}_2$.
\end{definition}

The following lemma is an obvious fact.
\begin{lemma}
\label{lemma3.1}
Let $\psi: \mathcal{G}\rightarrow \mathcal{E}$ be a factor map between continuous bundle RDS $T_1$ on $\mathcal{E}$ and $T_2$ on $\mathcal{G}$ over $\vartheta$, where $\mathcal{E}\subset \Omega \times X$, $\mathcal{G}\subset \Omega \times Z$. If $\mu \in \mathcal{I}_P(\mathcal{G})$, $\nu=\psi \mu$, $\mathcal{R}\in \mathcal{P}_{\mathcal{E}}$ and $\mathcal{U}\in \mathcal{C}^o_{\mathcal{E}}$, then
\renewcommand{\theenumi}{(\arabic{enumi})}
\begin{enumerate}
\item $h_{\text{top}}(T_2,\psi^{-1}(\mathcal{U}))=h_{\text{top}}(T_1,\mathcal{U})$;
\item $h^{(r)}_{\mu}(T_2,\psi^{-1}(\mathcal{R}))=h^{(r)}_{\nu}(T_1,\mathcal{R})$.
\end{enumerate}
\end{lemma}

We also need the following lemma.

\begin{lemma}
\label{lemma3.3}
Let $T$ be a continuous bundle RDS on $\mathcal{E}$ over $\vartheta$ and $\mathcal{R}\in \mathcal{P}_{\mathcal{E}}$. Then the following hold:
\renewcommand{\theenumi}{(\arabic{enumi})}
\begin{enumerate}
\item the function $\mu\rightarrow H_{\mu}(\mathcal{R}\mid \mathcal{F}_{\mathcal{E}})$ is concave on $\mathcal{P}_P(\mathcal{E})$;
\item the function $\mu\rightarrow h^{(r)}_{\mu}(T,\mathcal{R})$ and $\mu\rightarrow h^{(r)}_{\mu}(T)$ are affine on $\mathcal{I}_P(\mathcal{E})$.
\end{enumerate}
\end{lemma}

\begin{proof}
(1) Let $\mu=a\nu+(1-a)\eta$, where $\nu$, $\eta\in \mathcal{P}_P(\mathcal{E})$ and $0<a<1$. Since $H_{\mu}(\mathcal{R}\mid \mathcal{F}_{\mathcal{E}})=\int H_{\mu_{\omega}}(\mathcal{R}(\omega))dP(\omega)$, $\mu \in \mathcal{P}_P(\mathcal{E})$, and $\mu_{\omega}\rightarrow H_{\mu_{\omega}}(\mathcal{R}(\omega))$ is concave on $\mathcal{M}(X)$, where $\mathcal{M}(X)$  is the set of Borel probability measures on $X$. It is easy to see that
\begin{equation}\label{ineq3.1}
H_{\mu}(\mathcal{R}\mid \mathcal{F}_{\mathcal{E}})\geq a H_{\nu}(\mathcal{R}\mid \mathcal{F}_{\mathcal{E}})+(1-a)H_{\eta}(\mathcal{R}\mid \mathcal{F}_{\mathcal{E}}).
\end{equation}
Then $\mu\rightarrow H_{\mu}(\mathcal{R}\mid \mathcal{F}_{\mathcal{E}})$ is concave on $\mathcal{P}_P(\mathcal{E})$.

(2) Let $\mu=a\nu+(1-a)\eta$, where $\nu$, $\eta\in \mathcal{I}_P(\mathcal{E})$ and $0<a<1$. Using inequality \eqref{ineq3.1} we have
\begin{align*}
0&\leq H_{\mu}(\mathcal{R}\mid\mathcal{F}_{\mathcal{E}})
-aH_{\nu}(\mathcal{R}\mid\mathcal{F}_{\mathcal{E}})-(1-a)H_{\eta}(\mathcal{R}\mid\mathcal{F}_{\mathcal{E}})\\
&=\int (H_{\mu_{\omega}}(\mathcal{R}(\omega))
-aH_{\nu_{\omega}}(\mathcal{R}(\omega))-(1-a)H_{\eta_{\omega}}(\mathcal{R}(\omega)))dP(\omega)\\
&\leq \int (-a\log a - (1-a)\log (1-a))dP(\omega)\\
&=-a\log a - (1-a)\log (1-a).
\end{align*}
Hence
\begin{equation}\label{eq3.1}
h_{\mu}^{(r)}(T,\mathcal{R})=ah_{\nu}^{(r)}(T,\mathcal{R})+(1-a)h_{\eta}^{(r)}(T,\mathcal{R}),
\end{equation}
so that $\mu\rightarrow h_{\mu}^{(r)}(T,\mathcal{R})$ is affine.

Note that the supremum in the definition of $h_{\mu}^{(r)}(T)$ can be taken over partitions $\mathcal{Q}$ of $\mathcal{E}$ into sets $Q_i$ with the form $Q_i=(\Omega\times P_i)\cap\mathcal{E}$, where $\mathcal{P}=\{P_i\}$ is a partition of $X$. We can take an increasing sequence of finite Borel partitions $\mathcal{P}_j$ of $X$ with $\text{diam}(\mathcal{P}_j)\rightarrow 0$. Then
$$
(\Omega \times \bigvee_{j=1}^{\infty}\mathcal{P}_j)\vee (\mathcal{F}\times X)=\mathcal{F}\otimes\mathcal{B}.
$$
It follows from  Lemma 1.6 in \cite{Kifer}) that
$$
h^{(r)}_{\mu}(T)=\lim_{j\rightarrow \infty}h^{(r)}_{\mu}(T,\mathcal{Q}_j),
$$
where $\mathcal{Q}_j=\{Q_{j_i}\}$, $Q_{j_i}=(\Omega\times P_{j_i})\cap\mathcal{E}$, and $\mathcal{P}_j=\{P_{j_i}\}$ is a finite Borel partition of $X$. Replacing $\mathcal{R}$ by $\mathcal{R}_j$ in the equality \eqref{eq3.1}, letting $j\rightarrow \infty$, one has
$$
h_{\mu}^{(r)}(T)=ah_{\nu}^{(r)}(T)+(1-a)h_{\eta}^{(r)}(T),
$$
and we complete the proof of Lemma \ref{lemma3.3}.
\end{proof}

A real-valued function $f$ defined on a compact metric space $Z$ is
called {\it upper semi-continuous }(for short u.s.c.) if one of the
following equivalent conditions holds:
\begin{enumerate}
\item $\limsup_{z'\rightarrow z}f(z')\leq f(z)$ for each $z\in Z$;
\item for each $r\in \mathbb{R}$, the set $\{z\in Z:f(z)\geq r\}$ is
closed.\label{usc2}
\end{enumerate}
By \ref{usc2}, the infimum of any family of u.s.c. functions is
again a u.s.c. one; both the sum and supremum of finitely many
u.s.c. functions are u.s.c. ones.

For each function $f$ on $\mathcal{E}$, which is measurable in
$(\omega, x)$ and continuous in $x \in \mathcal{E}_{\omega}$, let
\begin{equation*}
\| f \| =\int \|  f(\omega) \|_{\infty} \, dP, \quad
\text{where} \quad \|  f(\omega) \|_{\infty} = \sup_{x\in
\mathcal{E}_{\omega}}\mid f(\omega,x) \mid.
\end{equation*}
Let $\mathbf{L}_{\mathcal{E}}^1 (\Omega, \mathcal{C}(X))$ be  the
space of such functions $f$ with $\|f \| < \infty$. If we identify
$f$ and $g$ for $f, g \in
 \mathbf{L}_{\mathcal{E}}^1 (\Omega, \mathcal{C}(X))$ with $\| f-g\| =0$, then
$\mathbf{L}_{\mathcal{E}}^1 (\Omega, \mathcal{C}(X))$ is a Banach
space with the norm $\| \cdot \|$.

For $\mu$, $\mu_n\in \mathcal{P}_P(\mathcal{E})$, $n=1,2,\dots$,
one called that $\mu_n$ converges to $\mu $ if $\int f d\mu_n \rightarrow \int f d\mu $
as $n\rightarrow \infty$ for any $f\in \mathbf{L}_{\mathcal{E}}^1 (\Omega, \mathcal{C}(X))$.
This introduces a weak* topology in $\mathcal{P}_P(\mathcal{E})$.
It is well known that $\mathcal{P}_P(\mathcal{E})$ is compact in this weak* topology.
Moreover, the follow lemma holds (\cite{Kifer2001}).

\begin{lemma}\label{lem3}
For any $\nu_k\in \mathcal{P}_P(\mathcal{E})$, $k\in \mathbb{N}$, the set of limit points in the above weak* topology of the sequence
$$\mu_n=\frac{1}{n}\sum_{k=0}^{n-1}\Theta^k\nu_n  \quad  \text{as}\,\,\,  n\rightarrow \infty$$
is not empty and is contained in $\mathcal{I}_P(\mathcal{E})$.
\end{lemma}

The  following lemma shows that in the sense of the above weak* topology the random measure-theoretical entropy map with the $\sigma$-algebra $\mathcal{F}_{\mathcal{E}}$ is u.s.c.. The first part of it was already given in \cite{Kifer2001}.

\begin{lemma}
\label{lemma3.4}
Let $T$ be a continuous RDS on $\mathcal{E}$ over $\vartheta$. Let $\mathcal{P}=\{P_1,\dots,P_k\}$ be a finite partition of $X$ satisfying $\int\mu_{\omega}(\partial\mathcal{P}_{\omega})dP(\omega)=0$, where $\mu_{\omega}$ are the disintegrations of $\mu$ and $\partial\mathcal{P}_{\omega}=\bigcup_{i=1}^k\partial (P_i\cap \mathcal{E}_{\omega})$ is the boundary of $\mathcal{P}_{\omega}=\{P_1\cap\mathcal{E}_{\omega},\dots,P_k\cap\mathcal{E}_{\omega}\}$; denote by $\mathcal{R}$ the partition of $\mathcal{E}$ into sets $(\Omega\times P_i)\cap \mathcal{E}$; then
\renewcommand{\theenumi}{(\alph{enumi})}
\begin{enumerate}
\item $\mu\rightarrow H_{\mu}(\mathcal{R}\mid \mathcal{F}_{\mathcal{E}})$ is a u.s.c. function on $\mathcal{P}_P(\mathcal{E})$.\label{a}
\item $\mu\rightarrow h_{\mu}^{(r)}(T,\mathcal{R})$  is a u.s.c. function on $\mathcal{I}_P(\mathcal{E})$.\label{b}
\end{enumerate}
\end{lemma}
\begin{proof}
We only prove the second part. By \ref{a}, $\mu\rightarrow H_{\mu}(\bigvee_{i=0}^{n-1}(\Theta^i)^{-1}\mathcal{R}\mid \mathcal{F}_{\mathcal{E}})$ is also a u.s.c. function on $\mathcal{I}_P(\mathcal{E})$. Note that for $\mu \in \mathcal{I}_P(\mathcal{E})$, $h_{\mu}^{(r)}(T,\mathcal{R})=\inf_{n\geq 1}\frac{1}{n}H_{\mu}(\bigvee_{i=0}^{n-1}(\Theta^i)^{-1}\mathcal{R}\mid \mathcal{F}_{\mathcal{E}})$, i.e. the function $\mu\rightarrow h_{\mu}^{(r)}(T,\mathcal{R})$ is the infimum of the family of u.s.c. functions $\frac{1}{n}H_{\mu}(\bigvee_{i=0}^{n-1}(\Theta^i)^{-1}\mathcal{R}\mid \mathcal{F}_{\mathcal{E}})$ on $\mathcal{I}_P(\mathcal{E})$. By \ref{usc2} in the definition of the u.s.c. function, $\mu\rightarrow h_{\mu}^{(r)}(T,\mathcal{R})$  is a u.s.c. function on $\mathcal{I}_P(\mathcal{E})$.
\end{proof}

The following Lemma \ref{lemma3.5} is important in the argument of the variational inequality of random entropy for $h_{\mu}^{(r)+}$. We follow the idea  of Kifer \cite{Kifer2001} for constructing a  measurable family of maximal separated sets  on fibers in bundle RDS  and  that of  Huang {\it et al.} \cite{Huang2006} for tackling with the  local variational inequality in the deterministic dynamical system. Then following Misiurewicz's method, we could avoid a similar combinatorial lemma in \cite{Blanchard1997} as in the deterministic case and obtain the variational inequality of random entropy stated in the beginning of this section.

If $\mathcal{R}=\{R_i\}$ is a partition of $\mathcal{E}$,
then $\mathcal{Q}=\bigvee_{i=0}^{n-1}(\Theta^i)^{-1}\mathcal{R}$ (denote by $(\mathcal{R})_0^{n-1}$) is a partition consisting of sets $\{Q_i\}$ such that the corresponding partition $\mathcal{Q}(\omega)=\{Q_j(\omega)\}$, $Q_j(\omega)=\{x: (\omega,x)\in Q_j\}$ of $\mathcal{E}_{\omega}$ has the form $\mathcal{Q}(\omega)=\bigvee_{i=0}^{n-1}(T_{\omega}^i)^{-1}\mathcal{R}(\vartheta^i\omega)$, where $\mathcal{R}(\omega)=\{R_i(\omega)\}$, $R_i(\omega)=\{x\in \mathcal{E}_{\omega}:(\omega,x)\in R_i \}$ partitions  $\mathcal{E}_{\omega}$.
Let $A_1, A_2\in \mathcal{F}\times \mathcal{B}$ and each $A_i(\omega)$ be a closed subset of $\mathcal{E_{\omega}}$. It follows from \cite[Proposition \Rmnum{3}.13]{Castaing} that $\{(\omega,x_1,x_2): x_i\in A_i, \forall i\}$ belongs to the product $\sigma$-algebra $\mathcal{F}\times \mathcal{B}^2$ (as a graph of a measurable multifunction).

\begin{lemma}
\label{lemma3.5}
Let $X$ be zero-dimensional and $T$ be a continuous RDS on $\mathcal{E}$ over $\vartheta$, $\mathcal{U}\in \mathcal{C}_{\mathcal{E}}^o$.
Assume $K\in \mathbb{N}$ and $\{\mathcal{R}_l\}_{l=1}^K$ is a finite sequence of partitions of $\mathcal{E}$ finer than $\mathcal{U}$, where $\mathcal{R}_l=\{R_{l,i}\}$, $1\leq l\leq K$ and each $R_{l,i}(\omega)$ is clopen subset of $X$. Then for each $n\in \mathbb{N}$, there exists a family of maximal subsets $B_n(\omega)\subset \mathcal{E}_{\omega}$ with cardinality at least $[N(T,\omega,\mathcal{U},n)/K] $ such that each atom of $(\mathcal{R}_l)_0^{n-1}(\omega)$ contains at most one point of $B_n(\omega)$ for $1\leq l\leq K$, and depending measurably on $\omega$ in the sense that $B_n=\{(\omega,x):x\in B_n(\omega)\}\in \mathcal{F}\times \mathcal{B}$,
where $[N(T,\omega,\mathcal{U},n)/K]$ is the integer part of $N(T,\omega,\mathcal{U},n)/K$.
\end{lemma}

\begin{proof}
Let $n\in \mathbb{N}$. For any $x\in \mathcal{E}_{\omega}$ and $1\leq l\leq K$, let $A_{l,n}^{\omega}(x)$ be the atom of $(\mathcal{R}_l)_0^{n-1}(\omega)$ containing the point $x$. Then for any $x_1$ and $x_2$ in $\mathcal{E}_{\omega}$ and $1\leq l\leq K$, $ x_1 $ and $ x_2 $ are contained in the same atom of $(\mathcal{R}_l)_0^{n-1}(\omega)$  if and only if $A_{l,n}^{\omega}(x_1)=A_{l,n}^{\omega}(x_2)$. For convenience, we write $\mathcal{Q}_l=(\mathcal{R}_l)_0^{n-1}$ and $\mathcal{Q}_l(\omega)=\{Q_{l,j}(\omega)\}=(\mathcal{R}_l)_0^{n-1}(\omega)$.

For $q\in \mathbb{Z}^+$, set
\begin{align*}
D_q&=\{(\omega,x_1,\ldots,x_q):\omega\in \Omega, \,\, x_i\in \mathcal{E}_{\omega}, \,\,\forall i\},\\
E_q^n&=\{(\omega,x_1,\ldots,x_q)\in D_q: \,\, A_{l,n}^{\omega}(x_i)\neq A_{l,n}^{\omega}(x_j), \,\, \forall i\neq j, \,\, \forall l \},\\
E_{q,l}^n&=\{(\omega,x_1,\ldots,x_q)\in D_q: \,\, A_{l,n}^{\omega}(x_i)\neq A_{l,n}^{\omega}(x_j), \,\, \forall i\neq j \},\\
E_q^n(\omega)&=\{(x_1,\ldots,x_q): (\omega,x_1,\ldots,x_q)\in
E_q^n\}\\
E_{q,l}^n(\omega)&=\{(x_1,\ldots,x_q): (\omega,x_1,\ldots,x_q)\in
E_{q,l}^n\}.
\end{align*}
Observe that $D_q\in \mathcal{F}\times \mathcal{B}^q$ (\cite{Kifer2001}), where $\mathcal{B}^q$ is the product $\sigma$-algebra on the product $X^q$ of $q$ copies of $X$. $E_{q,l}^n$ can also be expressed as
\begin{align*}
E_{q,l}^n &=\{(\omega,x_1,\ldots,x_q)\in D_q: x_i\in Q_{l,r}(\omega),\,\,  x_j\in Q_{l,s}(\omega),\,\, \forall i\neq j,\,\, \forall r\neq s \}\\
&=\bigcup_{(r_1,\cdots,r_q)}\{(\omega,x_1,\ldots,x_q)\in D_q:  x_i\in Q_{l,r_i}(\omega),\,\,  \forall i  \},
\end{align*}
where the union takes over all the elements of the set $\{(r_1,\cdots,r_q)\in \mathbb{N}^q: 1\leq r_1< r_2<\cdots< r_q \leq \#\mathcal{Q}_l\}$.

Put
\begin{align*}
E_{q,l}^{n,r}&=\{(\omega,x_1,\ldots,x_q)\in D_q:  x_i\in Q_{l,r_i}(\omega),\,\,  \forall i  \},\\
E_{q,l}^{n,r}(\omega)&=\{(x_1,\ldots,x_q)\in D_q:  (\omega,x_1,\ldots,x_q)\in E_{q,l}^{n,r}\}.
\end{align*}
The set $E_{q,l}^{n,r}$ may be empty. Note that  each $Q_{l,r_i}\in \mathcal{F}\times \mathcal{B}$ and $Q_{l,r_i}(\omega)$ is closed subset of $\mathcal{E}_{\omega}$ by the continuity of the RDS $T$.
If $E_{q,l}^{n,r}$ is not an empty subset of $\Omega\times X^q$, then $E_{q,l}^{n,r}\in \mathcal{F}\times \mathcal{B}^q$ and $E_{q,l}^{n,r}(\omega)$ is a closed subset of $\mathcal{E}_{\omega}^q$. In particular, $E_q^n\in \mathcal{F}\times \mathcal{B}^q$ and $E_q^n(\omega)$ is also a closed subset of  $\mathcal{E}_{\omega}^q$.

Let $s_n(\omega) $ be the largest cardinality of $B_n(\omega)$ such that any element of $\mathcal{Q}_l(\omega)$ contains at most one point of $B_n(\omega)$. By Theorem \Rmnum{3}.23 in \cite{Castaing} it follows that
$$
\{\omega: s_n(\omega)\geq q\}=\{\omega: E_q^n(\omega)\neq \emptyset\}=\Pr\nolimits_{\Omega}E_q^n\in \mathcal{F},
$$
where $\Pr\nolimits_{\Omega}$ is the projection of $\Omega \times X^q$ to $\Omega$, and so $s_n(\omega)$ is measurable in $\omega$.

Observe that the sets $\Omega_q=\{\omega: s_n(\omega)=q\}$ are measurable, disjoint and  $\bigcup_{q\geq 1}\Omega_q=\Omega$. It follows from Theorem \Rmnum{3}.30 in \cite{Castaing} that the multifunction $\Psi_q$ defined by $\Psi_q(\omega)=E_q^n(\omega)$ for $\omega\in \Omega_q$ is measurable, and it admits a measurable selection $\sigma_q$ which is a measurable map $\sigma_q: \Omega_q\rightarrow X^q$ such that $\sigma_q(\omega )\in E_q^n(\omega)$ for all $\omega\in \Omega_q$. Let $\zeta_q$ be the multifunction from $X^q$ to $q$-point subsets of $X$ defined by $\zeta_q(x_1,\cdots,x_q)=\{x_1,\cdots,x_q\}\subset X$. Then $\zeta_q\circ\sigma_q$ is a multifunction assigning to each $\omega\in \Omega_q$ a maximal  subset $B_n(\omega)$ in $\mathcal{E}_{\omega}$.

Let  $B'_n(\omega)$ be a maximal subset of $\mathcal{E}_{\omega}$ such that any atom of $\mathcal{Q}_l(\omega)$ contains at most  one point of $B'_n(\omega)$ for  $1\leq l \leq K$. We claim that the cardinality of $B'_n(\omega)$ is no less than  $[N(T,\omega,\mathcal{U},n)/K]$. Assume the contrary, i.e., $B'_n(\omega)=\{x_1,\dots, x_d\}$ with $d< [N(T,\omega,\mathcal{U},n)/K]$.

Set $B_{\omega}=\bigcup_{i=1}^d\bigcup_{l=1}^k A_{l,n}^{\omega}(x_i)\bigcap \mathcal{E}_{\omega}$. For any $1\leq i \leq d$ and $1\leq l\leq K$, $A_{l,n}^{\omega}(x_i)$ is an atom of $\mathcal{Q}_l(\omega)$, and thus is contained in an element of
$\mathcal{U}_0^{N-1}(\omega) =\bigvee_{i=0}^{N-1}(T^i_{\omega})^{-1}\mathcal{U}(\vartheta^i\omega)$.
Particularly, $B_{\omega}$ is covered by at most $dK<K[N(T,\omega,\mathcal{U},n)/K]\leq N(T,\omega, \mathcal{U}, n)$ elements of $\mathcal{U}_0^{N-1}(\omega)$ . Since any subcover of $\mathcal{U}_0^{N-1}(\omega)$ which covers $\mathcal{E}_{\omega}$ has at least $N(T,\omega,\mathcal{U},n)$ elements, we have $\mathcal{E}_{\omega}\backslash B_{\omega}\neq \emptyset$.

Choosing an arbitrary point $x\in \mathcal{E}_{\omega}\backslash B_{\omega}$, we have $ x \not\in \bigcup_{i=1}^d\bigcup_{l=1}^k A_{l,n}^{\omega}(x_i) $. Note that for any $1\leq i \leq d$ and $1\leq l\leq K$, $A_{l,n}^{\omega}(x)\neq A_{l,n}^{\omega}(x_i)$, we conclude that $B'_n(\omega)\cup \{x\}$ is also a subset of $\mathcal{E}_{\omega}$ such that any atom of $\mathcal{Q}_l(\omega)$ contains  at most one point of $B'_n(\omega)\cup \{x\}$ for  $1\leq l\leq K$. This is contradiction, as $B'_n(\omega)$ is maximal. Choosing $B_n(\omega)\subset B'_n(\omega)$ with the cardinality $[N(T,\omega,\mathcal{U},n)/K]$ we obtain the maximal subset we needed.

For any open subset $U\subset X$, set $V_U^q(i)=\{(x_1,\cdots,x_q)\in X^q: x_i\in U\}$ which is an open subset of $X^q$. Then
$$
\{\omega\in \Omega_q: \zeta_q\circ\sigma_q(\omega)\cap U\neq \emptyset\}=\bigcup_{i=1}^q\sigma_q^{-1}V_U^q(i)\in \mathcal{F}.
$$
Let $\Phi(\omega)=\zeta_{s_n(\omega)}\circ\sigma_{s_n(\omega)}(\omega)$, then
$$
\{\omega: \Phi(\omega)\cap U\neq \emptyset\}=\bigcup_{q=1}^{\infty}\{\omega\in \Omega_q: \zeta_q\circ\sigma_q(\omega)\cap U\neq \emptyset\}\in  \mathcal{F}.
$$
Hence $\Phi$ is a measurable multifunction which assigns to each $\omega\in \Omega$ a maximal finite subset $B_n(\omega)$ with cardinality at least $[N(T,\omega,\mathcal{U},n)/K] $ such that each atom of $(\mathcal{R}_l)_0^{n-1}(\omega)$ contains at most one point of $B_n(\omega)$ for $1\leq l\leq K$, and we complete the proof of Lemma \ref{lemma3.5}.
\end{proof}

\begin{lemma}
\label{lem2}
Let $T$ be a continuous bundle RDS on $\mathcal{E}$ over $\vartheta$. There exists a continuous bundle RDS $S$ on $Y\subset \Omega\times K^{\mathbb{N}}$ over $\vartheta$, where $K$ is a Cantor space, and a family of subjective continuous maps $\{\pi_{\omega}: Y_{\omega}\rightarrow \mathcal{E}_{\omega}\}_{\omega\in \Omega}$ such that $\pi_{\vartheta\omega}\circ S_{\omega}=T_{\omega}\circ \pi_{\omega}$ for $P$-a.s. $\omega$ and $\pi:(\omega,y)\rightarrow (\omega,\pi_{\omega}y)$ constitutes a measurable map from $Y$ to $\mathcal{E}$.
\end{lemma}

\begin{proof}
Since $X$ is compact metric space, there exists a Cantor space $K$ and  a subjective continuous map $f: K\rightarrow X$.
For  each $\omega\in \Omega$, $f^{-1}(\mathcal{E}_{\omega})$ is a closed subset of $K$. Let $K_{\omega}=f^{-1}(\mathcal{E}_{\omega})$ and $f_{\omega}$ be the restriction of $f$ on $K_{\omega}$. Denote by $\Pi_{i=0}^{\infty}K_{\vartheta^i\omega}=K_{\omega}\times K_{\vartheta\omega}\times\cdots \times K_{\vartheta^n\omega}\times \cdots$. Since $K_{\omega}$ is a Cantor subset of $K$, $\Pi_{i=0}^{\infty}K_{\vartheta^i\omega}$ is also a Cantor subspace of $ K^{\mathbb{N}}$, where the latter is equipped with the product topology.

For each $\omega$, put
\begin{equation*}
Y_{\omega}=\{y\in \Pi_{i=0}^{\infty}K_{\vartheta^i\omega}: T_{\vartheta^n\omega}f_{\vartheta^n\omega}(y_n)=f_{\vartheta^{n+1}\omega}(y_{n+1}) \,\, \text{for every}\,\, n\in\mathbb{N}         \}.
\end{equation*}
Then $Y_{\omega}$ is a closed subset of $K^{\mathbb{N}}$ for each $\omega\in \Omega$. Let $Y=\{(\omega,y): \omega\in \Omega, y\in Y_{\omega}\}$. Then $Y$ is measurable in $\Omega \times K^{\mathbb{N}}$, i.e.,  $\{\omega: Y_{\omega}\cap U\neq \emptyset\}\in \mathcal{F}$ for each open subset $U\in K^{\mathbb{N}}$. In fact, Let $U=\Pi_{i=0}^{\infty}U_i$, where $U_i$ is the element of the basis of $K$. Note that for each $i\in \mathbb{N}$,  $U_i$ is a clopen set. If $y\in Y_{\omega}\cap U $, then for each $n\in \mathbb{N}$, one have $y_n\in U_n$, $y_{n+1}\in U_{n+1}$ and $T_{\vartheta^n\omega}f_{\vartheta^n\omega}(y_n)=f_{\vartheta^{n+1}\omega}(y_{n+1})$.
Let $V_n=f_{\vartheta^n\omega}(U_n)=f(K_{\vartheta^n\omega}\cap U_n)$, $n\in \mathbb{N}$. Then $V_n$ is the closed subset of $\mathcal{E}_{\vartheta^n\omega}$, $n\in \mathbb{N}$. It follows that
\begin{align*}
\{\omega: Y_{\omega}\cap U\neq \emptyset\}
=&\bigcap_{n=0}^{\infty}
\{\omega: T_{\vartheta^n\omega}f_{\vartheta^n\omega}(y_n)=f_{\vartheta^{n+1}\omega}(y_{n+1}), y_n\in U_n, y_{n+1}\in U_{n+1}
\}\\
=&\bigcap_{n=0}^{\infty}\{\omega: T_{\vartheta^n\omega}x_n=x_{n+1}, x_n\in V_n, x_{n+1}\in V_{n+1} \}\\
=&\bigcap_{n=0}^{\infty}\{\omega: T_{\vartheta^n\omega}^{-1}V_{n+1}\cap V_n\neq \emptyset  \}
\end{align*}
Since the map $(\omega,x)\rightarrow T_{\omega}x$ is measurable and $V_n$ is closed in X for each $n\in \mathbb{N}$, then
\begin{align*}
 &\{(\omega,x)\in \Omega\times V_n: T_{\vartheta^n\omega}x\in V_{n+1}\}\\
= &\{(\omega,x): T_{\vartheta^n\omega}x\in V_{n+1}\}\cap \{(\omega,x): \omega\in \Omega, x\in V_n\}                   \end{align*}
is a measurable subset of $\Omega\times X$ for each $n \in \mathbb{N}$. By the projection theorem in \cite{Castaing} (Theorem \Rmnum{3}.23), one have
$$\{\omega: T_{\vartheta^n\omega}^{-1}V_{n+1}\cap V_n\neq \emptyset  \}\in \mathcal{F},$$
for each $n \in \mathbb{N}$. Then $\{\omega: Y_{\omega}\cap U\neq \emptyset\}\in \mathcal{F}$.

For each $\omega$, let $\pi_{\omega}: Y_{\omega}\rightarrow \mathcal{E}_{\omega}$ be defined by $\pi_{\omega}(y)=f_{\omega}(y_0)$. Then $\pi_{\omega}$ is a subjective continuous map.
Let $\pi: (\omega,y)\rightarrow (\omega,\pi_{\omega}y)$ be the map from $Y$ to $\mathcal{E}$.
Note that for fixed $y$, $\{\omega: \pi_{\omega}y\in A\}=Y_y$ or the null set for any open subset $A$ of $X$, where  $Y_y=\{\omega: (\omega,y)\in Y \}$ is the $y$-section of $Y$.
By the measurability of  $Y$,
one knows that the map $(\omega,y)\rightarrow \pi_{\omega}y$ is measurable in $\omega$. Since $(\omega,y)\rightarrow \pi_{\omega}y$ is continuous in $y$, then the map $(\omega,y)\rightarrow \pi_{\omega}y$ in jointly measurable and $\pi$ constitutes a measurable map from from $Y$ to $\mathcal{E}$.

Let $S_{\omega}: Y_{\omega}\rightarrow Y_{\vartheta\omega}$ be defined by the left shift $(S_{\omega}y)_i=y_{i+1}$, $i\in \mathbb{N}$ for each $\omega$ . Obviously, $(\omega,y)\rightarrow S_{\omega}y$ is continuous in $y$. With a similar argument as in the above proof of the measurability of $Y$, one can show that $(\omega,y)\rightarrow S_{\omega}y$  is measurable in $\omega$. Then the map $(\omega,y)\rightarrow S_{\omega}y$ is jointly measurable. Then the map $S:Y\rightarrow Y$ defined by $(\omega,y)\rightarrow (\vartheta\omega, S_{\omega}y)$ constitutes a skew product transformation. It is immediate to check that $\pi_{\vartheta\omega}\circ S_{\omega}=T_{\omega}\circ \pi_{\omega}$.
This completes the proof the lemma.
\end{proof}

\begin{remark}
In the deterministic dynamical system, it is well-known that for any dynamical system $(X,\varphi)$, where $X$ is a compact metric space and $\varphi:X\rightarrow X$ is a subjective continuous map, there exists a zero-dimensional dynamical system $(Z,\psi)$ and a subjective continuous map $\tau:Z\rightarrow X$ with $\tau\circ \psi=\varphi\circ \tau$ (See e.g. \cite{Blanchard1997}). The above lemma gives a random version of this result. When $(\Omega, \mathcal{F}, P, \vartheta)$ is a trivial system, Lemma \ref{lem2}  is the result in the deterministic case. For homeomorphic bundle RDS,   through replacing the left shift by two-sided shift and $ K^{\mathbb{N}}$ by $K^{\mathbb{Z}} $ one can find that a similar result also holds.
\end{remark}

The sequel of this section is devoted to the proof of the following variational inequality of random entropy for $h_{\mu}^{(r)+}$.

\begin{theorem}
\label{theorem3.6}
Let $T$ be a homeomorphic bundle RDS on $\mathcal{E}$ over $\vartheta$ and $\mathcal{U}\in \mathcal{C}_{\mathcal{E}}^{o'}$. Then there exists a $\mu\in \mathcal{I}_P(\mathcal{E})$ such that $h_{\mu}^{(r)+}(T,\mathcal{U})\geq h_{\text{top}}(T,\mathcal{U})$.
\end{theorem}

\begin{proof}

We adopt the argument in \cite{Huang2006} for the deterministic dynamical system.

Let $\mathcal{U}=\{\Omega\times U_i\}_{i=1}^d\cap \mathcal{E}\in \mathcal{C}_{\mathcal{E}}^{o'}$. Then $A=\{U_i\}_{i=1}^d$ is an open cover of $X$. Define
\begin{equation*}
\mathcal{U}^*=\{\mathcal{R}\in \mathcal{P}_{\mathcal{E}}: \mathcal{R}=\{\Omega\times R_i\}_{i=1}^d\cap\mathcal{E}, R_m\subset U_m \,\, \text{for all}\,\, 1\leq m\leq d         \}.
\end{equation*}

Case 1. Assume that $X$ is zero-dimensional. Then the family of partitions finer than $\mathcal{U}$, consisting of sets $\{\Omega\times R_i\}\cap\mathcal{E}$ with each $R_i$ being clopen sets, is countable. Let $\{\mathcal{R}_l:l\geq 1\}$ be the enumeration of this family. It is clear that $\{\mathcal{R}_l\}\subset \mathcal{U}^*$.

Fix $n\in \mathbb{N}$.  By Lemma \ref{lemma3.5}, there exists a measurable in $\omega$ family of  maximal subsets $C_n(\omega)$ of $\mathcal{E_{\omega}}$
$$
C_n(\omega)=B_{n^2}\big( (\Theta^n)^{-1}\mathcal{U},
\{(\Theta^n)^{-1}\mathcal{R}_l\}_{l=1}^n,\omega \big)
\subset \mathcal{E}_{\omega}
$$
with cardinality at least $\big[N(T,\omega,(\Theta^n)^{-1}\mathcal{U},n^2)/n\big]$ such that any atom of $(\mathcal{R}_l)_n^{n^2+n-1}(\omega)$ contains at most one point of $C_n(\omega)$ for all $1\leq l\leq n$.
Next, define probability measures $\nu_n$ on $\mathcal{E}$ via their measurable disintegrations $\nu_{n,\omega}$, where $\nu_{n,\omega}$ is the equidistributed probability measure on $C_n(\omega)$,
so that $d\nu_n(\omega,x)=d\nu_{n,\omega}(x)dP(\omega)$. Then for each $0\leq i, l\leq n$, every element of
$(\Theta^i)^{-1}(\mathcal{R}_l)_0^{n^2+n-1}(\omega)
=(T_{\omega}^i)^{-1}\bigvee_{j=0}^{n^2+n-1}(T_{\omega}^j)^{-1}\mathcal{R}_l(\vartheta^{j+i}\omega)$, which is finer than
$(\mathcal{R}_l)_n^{n^2+n-1}(\omega)
=\bigvee_{j=n}^{n^2+n-1}(T_{\omega}^j)^{-1}\mathcal{R}_l(\vartheta^j\omega)$, also contains at most one atom of the discrete measure $\nu_{n,\omega}$.
Since for each $\omega$,
\begin{align*}
N(T,\omega,\mathcal{U},n^2+n)
&\leq N(T, \omega,\mathcal{U},n)\cdot N(T,\omega,(\Theta^n)^{-1}\mathcal{U},n^2)\\
&\leq d^n N(T,\omega,(\Theta^n)^{-1}\mathcal{U},n^2),
\end{align*}
we have for any $1\leq i,  l\leq n$,
\begin{align}
\label{eq3.2}
H_{T_{\omega}^i\nu_{n,\omega}}&(\bigvee_{j=0}^{n^2+n-1}(T^j_{\omega})^{-1}\mathcal{R}_l(\vartheta^{j+i}\omega)
)
=H_{\nu_{n,\omega}}((T_{\omega}^i)^{-1}\bigvee_{j=0}^{n^2+n-1}(T^j_{\omega})^{-1}\mathcal{R}_l(\vartheta ^{j+i}\omega))\notag\\
&\geq\log [\frac{1}{n}N(T,\omega,(\Theta^n)^{-1}\mathcal{U},n^2)]\geq \log[\frac{1}{nd^n}N(T,\omega,\mathcal{U},n^2+n)].
\end{align}

Since $\nu_{n,\omega}$ is supported by $\mathcal{E}_{\omega}$, $T_{\omega}^i\nu_{n,\omega}$ is supported by $\mathcal{E}_{\vartheta^i\omega}$,
for all $1\leq i,l\leq n$, integrating in \eqref{eq3.2} against $P$,  we have by \eqref{kifer2.1} the inequality
\begin{align*}
H_{\Theta^i\nu_n}((\mathcal{R}_l)_0^{n^2+n-1}\mid \mathcal{F}_{\mathcal{E}})
&=H_{\nu_n}\big((\Theta^i)^{-1}(\mathcal{R}_l)_0^{n^2+n-1}\mid \mathcal{F}_{\mathcal{E}}\big)\\
&\geq \int \log[\frac{1}{nd^n}N(T,\omega,\mathcal{U},n^2+n)]dP(\omega).
\end{align*}

Fix $m\in \mathbb{N}$ with $m\leq n$, and let $n^2+n=km+b$, where $0\leq b\leq m-1$. Then for $1\leq i,l\leq n$, we have
\begin{align*}
&H_{\Theta^i\nu_n}((\mathcal{R}_l)_0^{n^2+n-1}\mid \mathcal{F}_{\mathcal{E}})\\
&=H_{\Theta^i\nu_n}\big((\mathcal{R}_l)_{km}^{n^2+n-1}\vee \bigvee_{j=0}^{k-1}(\Theta^{mj})^{-1}(\mathcal{R}_l)_0^{m-1}\mid \mathcal{F}_{\mathcal{E}}\big)\\
&\leq H_{\Theta^i\nu_n}((\mathcal{R}_l)_{km}^{n^2+n-1}\mid \mathcal{F}_{\mathcal{E}})+\sum_{j=0}^{k-1}H_{\Theta^i\nu_n}\big( (\Theta^{mj})^{-1} (\mathcal{R}_l)_0^{m-1} \mid \mathcal{F}_{\mathcal{E}}\big)\\
&\leq \sum_{j=0}^{k-1}H_{\Theta^{i+mj}\nu_n}((\mathcal{R}_l)_0^{m-1} \mid \mathcal{F}_{\mathcal{E}})+m\log d \quad(\text{by Lemma \ref{lemma2.2} (1) and (4)})
\end{align*}

Summing over $0\leq i\leq m-1$  for each $1\leq l\leq n$, we get
\begin{align*}
m\int\log[\frac{1}{nd^n} & N(T,\omega,\mathcal{U},n^2+n)]dP(\omega)
\leq \sum_{i=0}^{m-1}H_{\Theta^i\nu_n}((\mathcal{R}_l)_0^{n^2+n-1}\mid \mathcal{F}_{\mathcal{E}})\\
&\leq \sum_{i=0}^{m-1}\sum_{j=0}^{k-1}H_{\Theta^{i+jm}\nu_n}((\mathcal{R}_l)_0^{m-1}\mid \mathcal{F}_{\mathcal{E}}  )+m^2\log d\\
&\leq \sum_{i=0}^{n^2+n-1}H_{\Theta^i\nu_n}((\mathcal{R}_l)_0^{m-1}\mid \mathcal{F}_{\mathcal{E}}  )+m^2\log d.
\end{align*}

Denote
$$\mu_n=\frac{1}{n^2+n}\sum_{i=0}^{n^2+n-1}\Theta^i\nu_n.$$
Since the function $\mu\rightarrow H_{\mu} ((\mathcal{R}_l)_0^{m-1}\mid \mathcal{F}_{\mathcal{E}}  )$ is concave on $\mathcal{P}_P(\mathcal{E})$, by Lemma \ref{lemma3.3} part (1), we get for each $1\leq l\leq n$,
\begin{align}\label{ineq3.9}
H_{\mu_n} &((\mathcal{R}_l)_0^{m-1}\mid \mathcal{F}_{\mathcal{E}} )\geq \frac{1}{n^2+n}\sum_{i=0}^{n^2+n-1} H_{\Theta^i\nu_n}((\mathcal{R}_l)_0^{m-1}\mid \mathcal{F}_{\mathcal{E}} )\notag\\
&\geq \frac{m}{n^2+n}\big( \int\log[\frac{1}{nd^n} N(T,\omega,\mathcal{U},n^2+n)]dP(\omega)  -m\log d   \big).
\end{align}

Suppose that $\mu_{n_k}\rightarrow \mu$ as $k\rightarrow \infty$ with $\mu \in \mathcal{P}_P(\mathcal{E})$ in the weak* topology. Then by Lemma \ref{lem3}, $\mu \in \mathcal{I}_P(\mathcal{E})$. Fixing $m\in \mathbb{N}$, we have
\begin{equation}\label{eq3.10}
\lim_{k\rightarrow \infty}\frac{1}{n_k^2+n_k}\big( \int\log[\frac{1}{n_kd^{n_k}} N(T,\omega,\mathcal{U},n_k^2+n_k)]dP(\omega) -m\log d   \big)= h_{\text{top}}(T,\mathcal{U}).
\end{equation}

Since any element of $(\mathcal{R}_l)_0^{m-1}(\omega)$  is clopen for each $\omega\in \Omega$, by Lemma \ref{lemma3.4}, equation \eqref{eq3.10}, and replacing $n$ by $n_k$ in the inequality \eqref{ineq3.9} and letting $k\rightarrow +\infty$, we get
$$
\frac{1}{m}H_{\mu}((\mathcal{R}_l)_0^{m-1}\mid \mathcal{F}_{\mathcal{E}})\geq h_{\text{top}}(T,\mathcal{U}),
$$
for any $l, m\in \mathbb{N}$.

Fixing $l\in \mathbb{N}$ and letting $m\rightarrow \infty$, we get $h_{\mu}^{(r)}(T,\mathcal{R}_l)\geq h_{\text{top}}(T,\mathcal{U})$. Since $X$ is zero-dimensional, $\{\mathcal{R}_l\}_{l\geq 1}$ is dense in $\mathcal{U}^*$ with respect to the distance associated with $L^1(\mu)$. Hence
\begin{equation*}
h_{\mu}^{(r)+}(T,\mathcal{U})=\inf_{\mathcal{Q}\in \mathcal{P}_{\mathcal{E}}, \mathcal{Q}\succeq \mathcal{U}}h_{\mu}^{(r)}(T,\mathcal{Q})=\inf_{l\in \mathbb{N}}h_{\mu}^{(r)}(T,\mathcal{R}_l)\geq h_{\text{top}}(T,\mathcal{U}).
\end{equation*}
This completes the proof of Case 1.

Case 2. This is the general case.
By Lemma \ref{lem2}, there exists a homeomorphic bundle RDS $S$ on $\mathcal{G}\subset \Omega\times Z$ over $\vartheta$, where $Z$ is zero-dimensional, and a factor map $\psi:\mathcal{G}\rightarrow \mathcal{E}$ such that $T\circ \psi=\psi\circ S$.

By Case 1 and Lemma \ref{lemma3.1}, there exists $\nu\in \mathcal{I}_P(\mathcal{G})$ such that $h_{\nu}^{(r)+}(S,\psi^{-1}(\mathcal{U}))\geq h_{\text{top}}(S,\psi^{-1}(\mathcal{U}))=h_{\text{top}}(T,\mathcal{U})$. This means that for each partitions $\mathcal{R}$ of $\mathcal{G}$ finer than $\psi^{-1}(\mathcal{U})$, $h_{\nu}^{(r)}(S,\mathcal{R})\geq h_{\text{top}}(T,\mathcal{U})$. Let $\mu=\psi \nu$, then $\mu \in \mathcal{I}_P(\mathcal{E})$ \cite{Liupeidong2005}. It follows that for any partition $\mathcal{Q}$ of $\mathcal{E}$ finer than $\mathcal{U}$, $\psi^{-1}(\mathcal{Q})$ is also a partition of $\mathcal{G}$ finer than $\psi^{-1}(\mathcal{U})$. By Lemma \ref{lemma3.1} we have $h_{\mu}^{(r)}(T,\mathcal{Q})=h_{\nu}^{(r)}(S,\psi^{-1}(\mathcal{Q}))\geq h_{\text{top}}(T,\mathcal{U})$. Hence $h_{\mu}^{(r)+}(T,\mathcal{U})\geq h_{\text{top}}(T,\mathcal{U})$ and we complete the proof.
\end{proof}

\section{A local variational principle for $h_{\mu}^{(r)-}$}\label{section4}

In this section, we will prove a local variational principle for a certain class of covers of  $\Omega\times X$, i.e. for $\mathcal{U}\in \mathcal{C}_{\Omega\times X}^{o'}$. We first give some relations between the two kinds of  random measure-theoretical entropy for covers, then use these results and Theorem \ref{theorem3.6} to prove this local variational principle.

\begin{lemma}
\label{lemma4.1}
Let $T$ be a homeomorphic bundle RDS on $\mathcal{E}$ over $\vartheta$ and $\mu\in\mathcal{I}_P(\mathcal{E}) $. Then for $\mathcal{U}\in \mathcal{C}_{\mathcal{E}}$, we have
\renewcommand{\theenumi}{(\arabic{enumi})}
\begin{enumerate}
\item $h_{\mu}^{(r)-}(T,\mathcal{U})\leq h_{\mu}^{(r)+}(T,\mathcal{U})$;\label{411}
\item $h_{\mu}^{(r)-}(T,\mathcal{U})=(1/M)h_{\mu}^{(r)-}(T^M,\mathcal{U}_0^{M-1})$ for each $M\in \mathbb{N}$;\label{412}
\item $h_{\mu}^{(r)-}(T,\mathcal{U})=\lim_{n\rightarrow \infty}(1/n)h_{\mu}^{(r)+}(T^n,\mathcal{U}_0^{n-1})$.\label{413}
\end{enumerate}
\end{lemma}

\begin{proof}
Part \ref{411} and \ref{412} is obvious. We only prove part \ref{413}. By part \ref{411} and \ref{412}, for any $M\in \mathbb{N}$, we have
\begin{align*}
h_{\mu}^{(r)-}(T,\mathcal{U})&=\frac{1}{M}h_{\mu}^{(r)-}(T^M,\mathcal{U}_0^{M-1})\\
&\leq  \frac{1}{M}h_{\mu}^{(r)+}(T^M,\mathcal{U}_0^{M-1})\leq \frac{1}{M}H_{\mu}(\mathcal{U}_0^{M-1}\mid \mathcal{F}_{\mathcal{E}}).
\end{align*}
Taking the limit when $n\rightarrow +\infty$, we complete the argument.
\end{proof}

 For each $k\in \mathbb{N}$, let $\mathcal{I}_P(T^k,\mathcal{E})=\{\mu\in \mathcal{P}_P(\mathcal{E}):\mu \,\,\text{is}\,\, \Theta^k\text{-invariant}\}$. $\mathcal{I}_P(T,\mathcal{E})$ is simply written as the usual $\mathcal{I}_P(\mathcal{E})$.

Now we prove Theorem \ref{theorem2.5}.

\begin{proof}[Proof of Theorem \ref{theorem2.5}]
We follow the method applied in  \cite{Huang2006} and \cite{Huang2004} for the deterministic system.

Let $\mathcal{U}=\{\Omega\times U_i\}_{i=1}^d\cap\mathcal{E}\in \mathcal{C}_{\mathcal{E}}^{o'}$. Then $A=\{U_i\}_{i=1}^d$ is an open cover of $X$. Define
\begin{equation*}
\mathcal{U}^*=\{\mathcal{R}\in \mathcal{P}_{\mathcal{E}}: \mathcal{R}=\{\Omega\times R_i\}_{i=1}^d\cap\mathcal{E}, R_m\subset U_m \,\, \text{for all}\,\, 1\leq i\leq d         \}.
\end{equation*}

Case 1. Assume that $X$ is zero-dimensional. Then the family of partitions finer than $\mathcal{U}$, consisting of sets $\{\Omega\times R_i\}\cap\mathcal{E}$ with each $R_i$ being clopen sets, is countable. Let $\{\mathcal{R}_l:l\geq 1\}$ be the enumeration of this family. It is clear that $\{\mathcal{R}_l\}\subset \mathcal{U}^*$. Moreover, for any $k\in \mathbb{N}$ and $\mu\in\mathcal{I}_P(\mathcal{E})$,
\begin{equation}\label{eq4.1}
h_{\mu}^{(r)+}(T^k,\mathcal{U}_0^{k-1})=\inf_{s_k\in \mathbb{N}^k}h_{\mu}^{(r)}(T^k,\bigvee_{i=0}^{k-1}(\Theta^i)^{-1}\mathcal{R}_{s_k(i)}).
\end{equation}

For any $k\in \mathbb{N}$ and $s_k\in \mathbb{N}^k$, denote
\begin{align*}
M(k,s_k)=\big\{\mu \in \mathcal{I}_P(\mathcal{E}): &\frac{1}{k}h_{\mu}^{(r)} (T^k,\bigvee_{i=0}^{k-1}(\Theta^i)^{-1}\mathcal{R}_{s_k(i)})\\
         &\geq \frac{1}{k}h_{\text{top}}(T^k,\mathcal{U}_0^{k-1})=h_{\text{top}}(T,\mathcal{U})  \big\}.
\end{align*}
By Theorem \ref{theorem3.6} we know that there exists $\mu_k\in \mathcal{I}_P(T^k,\mathcal{E})$ such that
\begin{equation*}
h_{\mu_k}^{(r)+}(T^k,\mathcal{U}_0^{k-1})\geq h_{\text{top}}(T^k,\mathcal{U}_0^{k-1}).
\end{equation*}
Since $\bigvee_{i=0}^{k-1}(\Theta^i)^{-1}\mathcal{R}_{s_k(i)}$ is finer than $\mathcal{U}_0^{k-1}$ for each $s_k\in \mathbb{N}^k$, one has
\begin{equation}\label{ineq4.2}
h_{\mu_k}^{(r)}(T^k,\bigvee_{i=0}^{k-1}(\Theta^i)^{-1}\mathcal{R}_{s_k(i)})\geq h_{\text{top}}(T^k,\mathcal{U}_0^{k-1}).
\end{equation}

Let
 $$\nu_k=\frac{1}{k}\sum_{i=0}^{k-1}\Theta^i\mu_k.$$
As for each $0\leq i\leq k-1$, $\Theta^i\mu_k\in \mathcal{I}_P(T^k,\mathcal{E})$, one has $\nu_k\in \mathcal{I}_P(\mathcal{E})$. For $s_k\in \mathbb{N}^k$ and $1\leq j\leq k-1$, denote
\begin{equation*}
P^js_k=\underbrace{s_k(k-j)s_k(k-j+1)\dots s_k(k-1)}_j\underbrace{s_k(0)s_k(1)\dots s_k(k-j-1)}_{k-j}\in \mathbb{N}^k
\end{equation*}
and $P^0s_k=s_k$. It is easy to see that for $0\leq j\leq k-1$,
\begin{align*}
h_{\Theta^j\mu_k}^{(r)}(T^k, \bigvee_{i=0}^{k-1}(\Theta^i)^{-1}\mathcal{R}_{s_k(i)})
&=h_{\mu_k}^{(r)}(T^k,\bigvee_{i=0}^{k-1}(\Theta^i)^{-1}\mathcal{R}_{P^js_k(i)})\\
&\geq h_{\text{top}}(T^k,\mathcal{U}_0^{k-1}) \quad(\text{by inequality \eqref{ineq4.2}}).
\end{align*}
Moreover by Lemma \ref{lemma3.3} part (2) for each $s_k\in \mathbb{N}^k$,
\begin{align*}
h_{\nu_k}^{(r)}(T^k, \bigvee_{i=0}^{k-1}(\Theta^i)^{-1}\mathcal{R}_{s_k(i)})
&=\frac{1}{k}\sum_{i=0}^{k-1}h_{\Theta^i\mu_k}^{(r)}(T^k, \bigvee_{i=0}^{k-1}(\Theta^i)^{-1}\mathcal{R}_{s_k(i)})\\
&\geq h_{\text{top}}(T^k,\mathcal{U}_0^{k-1}).
\end{align*}
This means that $\nu_k\in \bigcap_{s_k\in \mathbb{N}^k}M(k,s_k)$. Let $M(k)=\bigcap_{s_k\in \mathbb{N}^k}M(k,s_k)$, Then $M(k)$ is non-empty subset of $\mathcal{I}_P(\mathcal{E})$.

By Lemma \ref{lemma3.4} part (b), for each $s_k\in \mathbb{N}^k$ the map $\mu\rightarrow h_{\mu}^{(r)}(T^k, \bigvee_{i=0}^{k-1}(\Theta^i)^{-1}\mathcal{R}_{s_k(i)})$ is a u.s.c. function from $\mathcal{I}_P(T^k,\mathcal{E})$ to $\mathbb{R}$. Since $\mathcal{I}_P(\mathcal{E})\subset \mathcal{I}_P(T^k,\mathcal{E})$, the map $\mu\rightarrow h_{\mu}^{(r)}(T^k, \bigvee_{i=0}^{k-1}(\Theta^i)^{-1}\mathcal{R}_{s_k(i)})$ is also u.s.c. on $\mathcal{I}_P(\mathcal{E})$. Therefore, $M(k,s_k)$ is closed in  $\mathcal{I}_P(\mathcal{E})$ for each $s_k\in \mathbb{N}^k$. Thus $M(k)$ is a non-empty closed subset of  $\mathcal{I}_P(\mathcal{E})$.

Now we show that if $k_1, k_2\in \mathbb{N}$ with $k_1\mid k_2$ then $M(k_2)\subset M(k_1)$. Let $\mu\in M(k_2)$ and $k=k_2/k_1$. For any $s_{k_1}\in \mathbb{N}^{k_1}$, let $s_{k_2}=s_{k_1}\dots s_{k_1}\in \mathbb{N}^{k_2}$. Then
\begin{align*}
h_{\text{top}}(T,\mathcal{U})
&\leq \frac{1}{k_2}h_{\mu}^{(r)}(T^{k_2},\bigvee_{i=0}^{k_2-1}(\Theta^i)^{-1}\mathcal{R}_{s_{k_2}(i)})\\
&=\frac{1}{k_1}\big[\frac{1}{k}h_{\mu}^{(r)}( T^{kk_1},  \bigvee_{j=0}^{k-1}(\Theta^{jk_1})^{-1}\bigvee_{i=0}^{k_1-1}(\Theta^i)^{-1}  \mathcal{R}_{s_{k_1}(i)}             ) \big]\\
&=\frac{1}{k_1}h_{\mu}^{(r)}(T^{k_1},\bigvee_{i=0}^{k_1-1}(\Theta^i)^{-1}\mathcal{R}_{s_{k_1}(i)}).
\end{align*}
Hence $\mu \in M(k_1,s_{k_1})$. Then $\mu \in M(k_1)$.

Since for each $k_1, k_2\in \mathbb{N}$, $M(k_1)\cap M(k_2)\supset M(k_1k_2)\neq \emptyset$, then $\bigcap_{k\in\mathbb{N}}M(k)\neq \emptyset$.
Take $\mu \in \bigcap_{k\in\mathbb{N}}M(k) $. For any $k\in \mathbb{N}$, by equation \eqref{eq4.1} one has
\begin{equation*}
\frac{1}{k}h_{\mu}^{(r)+}(T^k,\mathcal{U}_0^{k-1})=\inf_{s_k\in \mathbb{N}^k}\frac{1}{k}h_{\mu}^{(r)}(T^k, \bigvee_{i=0}^{k-1}(\Theta^i)^{-1}\mathcal{R}_{s_k(i)})\geq h_{\text{top}}(T,\mathcal{U}).
\end{equation*}
Moreover, by Lemma \ref{lemma4.1} part \ref{413}, one gets that
\begin{equation*}
h_{\mu}^{(r)-}(T,\mathcal{U})=\lim_{k\rightarrow +\infty}\frac{1}{k}h_{\mu}^{(r)+}(T^k, \mathcal{U}_0^{k-1})\geq h_{\text{top}}(T,\mathcal{U}).
\end{equation*}
Following Lemma \ref{lemma3.3} part (2), we have $h_{\mu}^{(r)-}(T,\mathcal{U})=h_{\text{top}}(T,\mathcal{U})$. This ends the proof of Case 1.

Case 2. This is the general case.

By Lemma \ref{lem2}, there exists a homeomorphic bundle RDS $S$ on $\mathcal{G}\subset \Omega\times Z$ over $\vartheta$, where $Z$ is zero-dimensional, and a factor map $\psi: \mathcal{G}\rightarrow \mathcal{E}$.  Let $\mathcal{V}=\psi^{-1}\mathcal{U}$. By Lemma \ref{lemma3.1} part (1), one has $h_{\text{top}}(S,\mathcal{V})=h_{\text{top}}(T,\mathcal{U})$. By Case 1, there exists $\nu\in \mathcal{I}_P(\mathcal{G})$ such that $h_{\mu}^{(r)-}(S,\mathcal{V})=h_{\text{top}}(S,\mathcal{V})$. Let $\mu=\psi\nu$, then $\mu\in \mathcal{I}_P(\mathcal{E})$. Note that if $N\in\mathbb{N}$ and $\mathcal{R}\in \mathcal{P}_{\mathcal{E}}$ is finer than $\mathcal{U}_0^{N-1}$, then $\psi^{-1}(\mathcal{R})\in \mathcal{P}_{\mathcal{G}}$ is finer than $\mathcal{V}_0^{N-1}$. Thus
$H_{\mu}(\mathcal{R}\mid \mathcal{F}_{\mathcal{E}})=H_{\nu}(\psi^{-1}\mathcal{R}\mid\psi^{-1}\mathcal{F}_{\mathcal{E}})
=H_{\nu}(\psi^{-1}\mathcal{R}\mid \mathcal{F}_{\mathcal{G}} )\geq H_{\nu}(\mathcal{V}_0^{N-1}\mid \mathcal{F}_{\mathcal{G}})$. Hence $ H_{\mu}(\mathcal{U}_0^{N-1}\mid \mathcal{F}_{\mathcal{E}})\geq H_{\nu}(\mathcal{V}_0^{N-1}\mid \mathcal{F}_{\mathcal{G}})$ for each $N\in \mathbb{N}$. Thus one get
\begin{align*}
h_{\mu}^{(r)-}(T,\mathcal{U})&=\lim_{N\rightarrow +\infty}\frac{1}{N}H_{\mu}(\mathcal{U}_0^{N-1}\mid \mathcal{F}_{\mathcal{E}})\geq \lim_{N\rightarrow +\infty}\frac{1}{N}H_{\nu}(\mathcal{V}_0^{N-1}\mid \mathcal{F}_{\mathcal{G}})\\
&=h_{\nu}^{(r)-}(S,\mathcal{V})=h_{\text{top}}(S,\mathcal{V})=h_{\text{top}}(T,\mathcal{U})\geq h_{\mu}^{(r)-}(T,\mathcal{U}).
\end{align*}
We complete the argument of this theorem.

\end{proof}

\end{document}